\documentclass[12pt]{amsart}
\usepackage{amsmath}
\usepackage{amsxtra}
\usepackage{amscd}
\usepackage{amsthm}
\usepackage{amsfonts}
\usepackage{amssymb}
\usepackage {pstricks}
\usepackage{pstricks,pst-node}
\usepackage[all]{xy}

\newtheorem{theorem}{Theorem}[section]

\newtheorem{lemma}[theorem]{Lemma}

\newtheorem{corollary}[theorem]{Corollary}
\newtheorem{proposition}[theorem]{Proposition}

\theoremstyle{definition}
\newtheorem{definition}[theorem]{Definition}
\newtheorem{example}[theorem]{Example}

\theoremstyle{remark}
\newtheorem{remark}[theorem]{Remark}

\numberwithin{equation}{section}

\newcommand{\be}{\begin{equation}}
\newcommand{\ee}{\end{equation}}

\newcommand{\C}{{\mathbb C}}

\newcommand{\Z}{{\mathbb Z}}

\newcommand{\id}{{\rm{id}}}

\newcommand{\mf}{\mathfrak}
\newcommand{\fg}{{\mf g}}

\newcommand{\fb}{{\mf b}}

\newcommand{\fp}{{\mathfrak p}}

\newcommand{\fu}{{\mathfrak u}}

\newcommand{\U}{{\rm{U}}}
\newcommand{\End}{{\rm{End}}}

\advance\headheight by 2pt

\newcommand{\fsl}{{\mathfrak {sl}}}

\newcommand{\fgl}{{\mathfrak {gl}}}

\newcommand{\gl}{{\mathfrak {gl}}}

\newcommand{\TU}{\widetilde{\U}}

\begin{document}

\normalfont

\title[Degenerate quantum groups]{Degenerate quantum general linear groups}

\author{Jin Cheng, Yan Wang}
\address[Cheng, Wang]{School of Mathematics and Statistics,
Shandong Normal University, Jinan, China}
\email{cheng934@mail.ustc.edu.cn}
\author{R. B. Zhang}
\address[Zhang]{School of Mathematics and Statistics,
University of Sydney, Sydney, Australia}
\email{ruibin.zhang@sydney.edu.au}
\date{\today}
\begin{abstract}
Given any pair of positive integers $m$ and $n$,   we construct a new Hopf algebra, which may be regarded as a  degenerate version of the quantum group of $\gl_{m+n}$.  We study its structure and develop a highest weight representation theory. The finite dimensional simple modules are classified in terms of highest weights, which are essentially characterised by $m+n-2$ nonnegative integers and two arbitrary nonzero scalars. In the special case with $m=2$ and $n=1$, an explicit basis  is constructed for each finite dimensional simple module.
For all $m$ and $n$, the degenerate quantum group has a natural irreducible representation acting on $\C(q)^{m+n}$. It admits an $R$-matrix that satisfies the Yang-Baxter equation and intertwines the co-multiplication and its opposite. This in particular gives rise to isomorphisms between the two module structures of any tensor power of $\C(q)^{m+n}$ defined relative to the co-multiplication and its opposite respectively.  A topological invariant of knots is constructed from this $R$-matrix, which reproduces the celebrated HOMFLY polynomial. Degenerate quantum groups of other classical types are briefly discussed.
\end{abstract}
\maketitle


\section{Introduction}\label{sect:intro}

Quantum groups \cite{D85, D86, J1, J2} emerged from the theory of Yang-Baxter type integrable models in statistical mechanics \cite{B} some 30 years ago.  Since then the study of these remarkable algebraic structures has grown into a vast research area \cite{CP, L}, with important applications in a number of areas in mathematical physics and pure mathematics such as integrable models \cite{B}, conformal field theory, low dimensional topology \cite{Jv, RT, ZGB, T} and representation theory \cite{L}.

The term quantum groups refers to a class of Hopf algebras and Hopf superalgebras, which include Drinfeld-Jimbo quantum groups \cite{D85, D86, J1, J2}, quantum supergroups \cite{BGZ, CK, Geer, Y, ZGB}, quantum affine Kac-Moody Lie algebras and superalgebras, and their dual Hopf (super)algebras \cite{FRT, Z98}, which are quantum analogues of algebraic (super) groups.
Research in the area largely focused on the structure and representation theory of these objects.
There were attempts to explore other quantum deformed algebras, such as multi-parameter quantum algebras, but with very limited success so far.  Presumably the reason is that  quantum groups are relatively rigid objects, thus  nontrivial tinkering of the definition would drastically change their structures.

In this paper, we investigate a class of new Hopf algebras which may be considered as degenerate versions of Drinfeld-Jimbo quantum groups of type $A$.  These Hopf algebras have interesting structures and a rich representation theory with useful applications in solving the Yang-Baxter equation and in constructing topological invariants of knots.

The idea goes back to an old paper of Cosmas K. Zachos appeared in the physics literature  \cite{Za},  where he studied symmetry properties of wave functions of quantum mechanical systems under the action of the Hopf algebra defined by \eqref{eq:Za}, which may be regarded as quantum $\fsl_2$ at $\sqrt{-1}$. 
The degenerate quantum groups in the present paper are obtained by letting one of the quantum $\fsl_2$ subalgebras of a quantum group degenerate to the algebra of Zachos' and appropriately modifying the Serre relations involving it, while keeping the other quantum $\fsl_2$ subalgebras essentially intact.  
Our investigation here is also partially motivated by a desire to understand quantum group analogues of Inon\"u-Wigner contractions. 

The immediate question is whether this leads to sensible Hopf algebras. By being sensible we mean that they should have structural properties which allow for a representation theory similar to that of universal enveloping algebras of semi-simple or affine Kac-Moody Lie (super)algebras.  In particular,  the degenerate quantum groups associated with finite root data should have enough finite dimensional irreducible representations at generic $q$ to make the representation theory interesting.
As we will see, at least in the type $A$ case, this indeed leads to interesting Hopf algebras.
Furthermore, Remark \ref{rem:no-deform} indicates that degenerate quantum groups obtained this way are not ``deformation quantisations" of universal enveloping algebras of  Lie algebras or Lie superalgebras.

We now briefly describe the content of this paper.

We construct a degenerate quantum general linear group $\U_q(\gl_{m, n})$ of the general linear Lie algebra $\gl_{m+n}$ for each pair of positive integers $m, n$. It is a Hopf algebra containing  a Hopf subalgebra $\U_q(\fsl_{m, n})$, which we call the degenerate quantum special linear group.

We develop the structure of $\U_q(\gl_{m, n})$ and apply results to study a highest weight representation theory.  A classification of the finite dimensional simple modules is obtained in terms of highest weights in Section \ref{sect:rep-theory}. Such modules are essentially characterised by elements of $\Z_+^{\times(m+n-2)}\times\C(q)^*\times\C(q)^*$ with $\C(q)^*=\C(q)\backslash\{0\}$,
see Theorem \ref{thm:fd-irreps}.
In the special case $m=2$ and $n=1$, an explicit basis  is constructed for each finite dimensional simple module in Lemma \ref{lem:typical} and Lemma \ref{lem:atypical}.

We endow $V=\C(q)^{m+n}$ with a module structure of
the degenerate quantum general linear group $\U_q(\gl_{m, n})$ for any $m, n$ (see Section \ref{sect:tensors}).  Tensor powers of $V$ lead to an infinite family of finite dimensional representations.
An immediate question is whether the $\U_q(\gl_{m, n})$-modules
with the same underlying vector space $V^{\otimes r}$ but defined relative to the co-multiplication and its opposite are isomorphic. We answer this question in the affirmative in Section \ref{sect:R-matrix}. This is done by constructing an $R$-matrix, that is, a solution of the Yang-Baxter equation, which intertwines the two co-multiplications.

We develop aspects of the invariant theory of $\U_q(\gl_{m, n})$ in Section \ref{sect:inv}, and apply them to the $R$-matrix mentioned above to construct a topological invariant of knots in Theorem \ref{thm:knot-inv}. The knot invariant obtained coincides with the celebrated HOMFLY polynomial \cite{HOMFLY}. This is an interesting and important application of $\U_q(\gl_{m, n})$.

It should be pointed out that the emergence of the solution of the Yang-Baxter equation (see Section \ref{sect:R-matrix}) and construction of the HOMFLY polynomial (see Section \ref{sect:app}) in the context of the representation theory of $\U_q(\gl_{m, n})$ clearly demonstrate that the degenerate quantum general group will have an important role  to play in soluble models and low dimensional topology. This fact alone justifies a thorough investigation of the structure and representations of $\U_q(\gl_{m, n})$.

We also briefly discuss how to generalise the definition of degenerate quantum general linear group to degenerate quantum groups of $B$, $C$ and $D$ types. This is done by introducing generalised Dynkin diagrams, and then defining a degenerate quantum group for each diagram, see Section \ref{sect:general}.
In a future work, we hope to develop a general framework for the degeneration process, which will enable us to systematically study degenerate quantum groups associated with finite dimensional simple Lie algebras and affine Kac-Moody Lie algebras.

Despite the fact that the degenerate quantum general linear group $\U_q(\gl_{m, n})$ is only an ordinary Hopf algebra (i.e., does not have an odd subspace), while the quantum general linear supergroup $\U_q(\gl_{m|n})$  \cite{Z93, Z98} is a  Hopf superalgebra, we observe a number of similarities between them, see Section \ref{sect:duality} for further discussions.  It will be very interesting to find a precise connection between the two, e.g., analogous to the quantum correspondences between quantum affine superalgebras studied in \cite{XZ, Z92, Z97}. 
If such a correspondence exists, it will enable us to study quantum supergroups from the new perspective of degenerate quantum groups.  This will be particularly welcome, as the theory of quantum supergroups is not very well understood. 

Throughout the paper,  we work over the field $\C(q)$ of rational functions in the indeterminate  $q$.

\section{A degeneration of the quantum group of $\fsl_3$}\label{sect:sl3}
We start with the example of a degenerate version of the quantum group of $\fsl_3$.
This example has  most of the new properties not shared by  ordinary quantum groups, and its analysis presented here applies to arbitrary degenerate quantum groups. Results obtained here will be used in an essential way in later sections.
The advantage of treating this example first is that one can easily isolate the new properties to gain a clear understanding of them.  

This degenerate quantum group of $\fsl_3$ is also the first nontrivial example of which
the structure and representation theory can be thoroughly understood.
It provides an ideal test ground for assessing whether the idea of degenerating quantum groups is likely to be fruitful.

\subsection{The degenerate quantum group $\U_q(\fsl_{2, 1})$}
Recall that the Drinfeld-Jimbo quantum group $\U_q(\fsl_3)$ is generated by two $\U_q(\fsl_2)$ subalgebras with the same $q$, which are linked together by certain relations, in particular, Serre relations.  The corresponding degenerate quantum group is obtained by keeping one quantum $\fsl_2$ subalgebra as it is at generic $q$, while letting the other degenerate to the Hopf algebra used by Zachos,
which may be thought \cite{Z92} as generated by
$k, k^{-1}, X^+, X^-$ with relations
\begin{eqnarray}\label{eq:Za}
\begin{aligned}
&k k^{-1}=1, \quad k X^\pm k^{-1} =  -X^\pm, \\
 &X^+X^- -  X^-X^+= \frac{k - k^{-1}}{q-q^{-1}}, \quad (X^\pm)^2=0.
\end{aligned}
\end{eqnarray}
[The Hopf algebraic structure of it will become clear.]
In more precise term, we have the following definition.
\begin{definition} \label{def:case21}
Let $\U_q(\fsl_{2, 1})$ be the unital associative algebra over $\C(q)$ defined by the following presentation. The generators are
$e_i, \ f_i, \ k_i, \ k_i^{-1}$ ($i=1, 2$)
and the relations are given by
\begin{eqnarray}
&k_i k_i^{-1} =1,  \quad   k_i^{\pm 1} k_j^{\pm 1} = k_j^{\pm 1} k_i^{\pm 1},  \quad \forall i, j,
\label{eq:relat-1}\\
&k_i e_j k_i^{-1} =  q^{-1} e_j,  \quad k_i f_j k_i^{-1} =  q f_j, \quad i\ne j,  \label{eq:relat-2}\\
&k_1 e_1 k_1^{-1} = q^2 e_1, \quad  k_1 f_1 k_1^{-1} = q^{-2} f_1, \label{eq:relat-3}\\
&k_2 e_2 k_2^{-1} = - e_2, \quad k_2 f_2 k_2^{-1} = - f_2, \label{eq:relat-4}\\
&e_i f_j - f_j e_i = \delta_{i j} \frac{k_i - k_i^{-1}}{q-q^{-1}}, \label{eq:relat-5}\\
&e_1^2 e_2 - (q+q^{-1}) e_1 e_2 e_1 + e_2 e_1^2 =0, \label{eq:relat-6}\\
&f_1^2 f_2 - (q+q^{-1}) f_1 f_2 f_1 + f_2 f_1^2 =0, \label{eq:relat-7}\\
&e_2^2=0, \quad f_2^2=0. \label{eq:relat-8}
\end{eqnarray}
Call $\U_q(\fsl_{2, 1})$ a degenerate quantum group of $\fsl_3$.
\end{definition}

\begin{remark}
The elements $k_2^{\pm 1}, e_2, f_2$ generate a subalgebra isomorphic to that defined by \eqref{eq:Za}. It may be regarded as a degenerate quantum $\fsl_2$ and we denote it by $\U_q(\fsl_{1,1})$.
\end{remark}

\begin{remark} \label{rem:no-deform}
Note that $\U_q(\fsl_{1,1})$ (defined over $\C[q, q^{-1}]$) does not specialise at $q=1$ to the universal enveloping algebra of any Lie algebra or Lie superalgebra. Hence it is not a
 ``deformation quantisation" of any universal enveloping algebra.
\end{remark}

\begin{remark}\label{rem:super}
It is interesting to compare the definition of $\U_q(\fsl_{2,1})$ with that of the quantum special linear supergroup $\U_q(\fsl_{2|1})$ \cite{BGZ, Geer, Y, ZGB}.  Their differences lie in \eqref{eq:relat-4} and \eqref{eq:relat-5}.
If one replaces \eqref{eq:relat-4} and \eqref{eq:relat-5} in Definition \ref{def:case21} by the following relations respectively,
\[
\begin{aligned}
&k_2 e_2 k_2^{-1} = e_2, \quad k_2 f_2 k_2^{-1} = f_2, \quad
&e_i f_j - (-1)^{[e_i][f_j]} f_j e_i = \delta_{i j} \frac{k_i - k_i^{-1}}{q-q^{-1}},
\end{aligned}
\]
where $[e_1]=[f_1]=0$ and $[e_2]=[f_2]=1$ are the parity of these elements, one obtains the quantum supergroup $\U_q(\fsl_{2|1})$ as an associative algebra.
Now $\U_q(\fsl_{2|1})$ is a  Hopf superalgebra; its $\Z_2$-grading enters the definition of the co-multiplication in a nontrivial way.  However, $\U_q(\fsl_{2,1})$ is only an ordinary (i.e., not super) Hopf algebra as we will see.
\end{remark}

\subsection{Hopf algebraic structure of $\U_q(\fsl_{2, 1})$}

We now consider the structure of the degenerate quantum group $\U_q(\fsl_{2, 1})$.

\begin{lemma} \label{lem:hopf-21} Let $\widetilde{\U}_q(\fsl_{2, 1})$ be the unital algebra generated by $e_i, \ f_i, \ k_i^{\pm 1}$ ($i=1, 2$) subject to the relations \eqref{eq:relat-1} to \eqref{eq:relat-5} only.
Then $\TU_q(\fsl_{2, 1})$  has the structure of a Hopf algebra with co-multiplication
$\Delta: \TU_q(\fsl_{2, 1}) \longrightarrow \TU_q(\fsl_{2, 1})\otimes \TU_q(\fsl_{2, 1})$
\[
\begin{aligned}
\Delta(e_i)=e_i\otimes k_i + 1\otimes e_i, \quad
\Delta(f_i)=f_i\otimes 1 + k_i^{-1}\otimes f_i, \quad
\Delta(k_i)=k_i\otimes k_i;
\end{aligned}
\]
co-unit $\epsilon:  \TU_q(\fsl_{2, 1}) \longrightarrow \C(q)$
\[
\epsilon(e_i)=\epsilon(f_i)=0, \quad \epsilon(k_i)=1;
\]
and antipode $S: \TU_q(\fsl_{2, 1}) \longrightarrow \TU_q(\fsl_{2, 1})$
\[
S(e_i)=- e_i k_i^{-1}, \quad S(f_i)= - k_i f_i, \quad S(k_i)=k_i^{-1}.
\]
\end{lemma}
\begin{proof}
(1).
The bulk of the proof is in showing that $\Delta$ is an algebra homomorphism from $\TU_q(\fsl_{2, 1})$ to $\TU_q(\fsl_{2, 1})\otimes \TU_q(\fsl_{2, 1})$,
where the multiplication of the latter is defined in the standard way: for all $a, a', b, b'$ in $\TU_q(\fsl_{2, 1})$,
\[
(a\otimes b)(a'\otimes b')=a a' \otimes b b'.
\]
[Note that no sign factors are required in contrast to the case of superalgebras.]
We will prove this by showing that $\Delta$ preserves the relations \eqref{eq:relat-1} -- \eqref{eq:relat-5}.  This is clear for \eqref{eq:relat-1}.
To consider \eqref{eq:relat-2}, we note that for $i\ne j$,
\[
\begin{aligned}
\Delta(k_i) \Delta(e_j)\Delta(k_i^{-1})
&=k_ie_jk_i^{-1}\otimes k_j+1\otimes k_ie_jk_i^{-1}\\
&=q^{-1}e_j\otimes k_j+1\otimes q^{-1}e_j =q^{-1}\Delta (e_j),
\end{aligned}
\]
and similarly
$\Delta(k_i)\Delta(f_j)\Delta(k_i^{-1})=q\Delta(f_j).$
The proofs for \eqref{eq:relat-3}  and \eqref{eq:relat-4} are the same, thus are omitted.

For \eqref{eq:relat-5},  we have
\[
\begin{aligned}
&\Delta(e_i)\Delta(f_j)-\Delta(f_j)\Delta(e_i)\\
&=(e_if_j-f_je_i)\otimes k_i+k_j^{-1}\otimes (e_if_j-f_je_i)\\
&+e_ik_j^{-1}\otimes k_if_j-k_j^{-1}e_i\otimes f_jk_i.
\end{aligned}
\]
The last two terms on the right hand side cancel, and by \eqref{eq:relat-5}, the remaining two terms can be rewritten as
\[
\begin{aligned}
\delta_{ij} \frac{k_i-k_i^{-1}}{q-q^{-1}}\otimes k_i+\delta_{ij}k_j^{-1}\otimes\frac{k_i-k_i^{-1}}{q-q^{-1}}
=\delta_{ij} \frac{k_i\otimes k_i-k_i^{-1}\otimes k_i^{-1}}{q-q^{-1}}.
\end{aligned}
\]
Hence
$\Delta(e_i)\Delta(f_j)-\Delta(f_j)\Delta(e_i)=\delta_{ij} \frac{\Delta(k_i)-\Delta(k_i^{-1})}{q-q^{-1}}.
$

(2). 
Since $\Delta$ is an algebra homomorphism, in order to show its co-associativity, we only need to prove it on the generators. This can be done by computation, e.g.,
\[
\begin{aligned}
(\id\otimes \Delta)\Delta(f_i) &=f_i\otimes 1 \otimes 1 + k_i^{-1}\otimes f_i \otimes 1  + k_i^{-1}\otimes k_i^{-1}\otimes f_i \\
&=(\Delta\otimes \id)\Delta(f_i).
\end{aligned}
\]

(3).
We can easily show  that the map $\epsilon$  is an algebra homomorphism.
Denote by $\mu: \TU_q(\fsl_{2, 1})\otimes\TU_q(\fsl_{2, 1})  \longrightarrow \TU_q(\fsl_{2, 1})$  the multiplication of $\TU_q(\fsl_{2, 1})$. It is clear that for $x=e_i, f_i, k_i, k_i^{-1}$ for all $i$,
the following relation holds.
\[
\mu(\id\otimes\epsilon)\Delta(x)= \mu(\epsilon\otimes\id)\Delta(x)=x.
\]
This then holds for all $x\in \TU_q(\fsl_{2, 1})$, since $\Delta$ is also an algebra homomorphism. Hence $\epsilon$ defines a co-unit.

(4).  It is easy to prove that $S$ is an algebra anti-automorphism by showing that
it  preserves the relations \eqref{eq:relat-1} - \eqref{eq:relat-5} but reversing the order of the products. For example,  for \eqref{eq:relat-5}, we have
\[
\begin{aligned}
S(e_i f_j - f_j e_i)&= S(f_j) S(e_i) - S(e_i)S(f_j)
= k_j f_j e_i k_i^{-1} - e_i k_i^{-1} k_j f_j \\
&=  k_j (f_j e_i  - e_i f_j) k_i^{-1}
= - \delta_{i j} k_j \frac{k_i - k_i^{-1}}{q-q^{-1}} k_i^{-1} \\
&= \delta_{i j} \frac{S(k_i) - S(k_i^{-1})}{q-q^{-1}}.
\end{aligned}
\]

Now we show that $S$ has the required properties of an antipode in relation to the co-multiplication $\Delta$ and co-unit $\epsilon$. It is easy to show that for $x=e_i, f_i, k_i$,
\begin{eqnarray}\label{eq:S-def}
\mu(S\otimes\id)\Delta(x)=\mu(\id\otimes S)\Delta(x)=\epsilon(x).
\end{eqnarray}
Since $S$ is an algebra anti-automorphism, and $\Delta$, $\epsilon$ are algebra homomorphisms, the above relations hold for all elements of $x\in \TU_q(\fsl_{2, 1})$.  Hence $S$ is the required antipode.

This completes the proof.
\end{proof}

Introduce the following elements of $\widetilde{\U}_q(\fsl_{2, 1})$:
\[
\begin{aligned}
&S_{1 2}^{(+)}=e_1^2 e_2 - (q+q^{-1}) e_1 e_2 e_1 + e_2 e_1^2, \\
&S_{1 2}^{(-)}=f_1^2 f_2 - (q+q^{-1}) f_1 f_2 f_1 + f_2 f_1^2,\\
&S_{2}^{(+)}= e_2^2, \quad S_2^{(-)}= f_2^2.
\end{aligned}
\]
\begin{lemma}\label{lem:Serre-relations}
The elements $S_{1 2}^{(\pm)}$ and $S_2^{(\pm)}$ satisfy the following relations
\begin{enumerate}
\item \[
\begin{aligned}
&f_i S_{1 2}^{(+)} -  S_{1 2}^{(+)}f_i =0, \quad  f_i S_2^{(+)} -  S_2^{(+)}f_i=0, \\
&e_i S_{1 2}^{(-)} -   S_{1 2}^{(-)}e_i =0, \quad  e_i S_2^{(-)} -   S_2^{(-)}e_i=0, \quad i=1, 2.
\end{aligned}
\]
\item
\[
\begin{aligned}
\Delta(S_{1 2}^{(+)})&=S_{12}^{(+)}\otimes k_1^2k_2 + 1\otimes S_{12}^{(+)}, \\
\Delta(S_{1 2}^{(-)})&=S_{1 2}^{(-)}\otimes 1 + k_1^{-2}k_2^{-1}\otimes S_{1 2}^{(-)}, \\
\Delta(S_2^{(+)})&=S_2^{(+)}\otimes k_2^2 + 1\otimes S_2^{(+)}, \\
\Delta(S_2^{(-)})&=S_2^{(-)} \otimes 1+ k_2^{-2}\otimes S_2^{(-)}.
\end{aligned}
\]
\item
\[
\epsilon(S_{1 2}^{(+)})=0, \ \epsilon(S_{1 2}^{(-)})=0, \  \epsilon(S_2^{(+)})=0, \   \epsilon(S_2^{(-)})=0.
\]
\end{enumerate}
\end{lemma}
\begin{proof}
(1).
For any elements $x, y$ in $\TU_q(\fsl_{2, 1})$, we write
$
[x, y] = x y -y x.
$
We have
\[
\begin{aligned}
{[f_1, S_{1 2}^{(+)}]}
=[f_1, e_1^2] e_2-(q+q^{-1})[f_1, e_1]e_2e_1 -(q+q^{-1})e_1e_2[f_1, e_1]+e_2 [f_1, e_1^2].
\end{aligned}
\]
The formula
\begin{eqnarray}\label{eq:f-e-power}
[f_1, e_1^k] = - [k]_q e_1^{k-1} \frac{k_1 q^{k-1} - k_1^{-1}q^{1-k}}{q-q^{-1}} \quad \text{with } \ [k]_q:=\frac{q^k - q^{-k}}{q-q^{-1}}
\end{eqnarray}
in the $k=2$ case leads to
$[f_1, e_1^2]=  - (q+q^{-1}) e_1 \frac{k_1 q - k_1^{-1}q^{-1}}{q-q^{-1}}$.
Hence
\[
\begin{aligned}
{[f_1, S_{1 2}^{(+)}]}
&= - (q+q^{-1}) e_1 \frac{k_1 q - k_1^{-1}q^{-1}}{q-q^{-1}} e_2
- (q+q^{-1}) e_2  e_1 \frac{k_1 q - k_1^{-1}q^{-1}}{q-q^{-1}}\\
&
+(q+q^{-1})  \frac{k_1 - k_1^{-1}}{q-q^{-1}} e_2e_1 + (q+q^{-1})e_1e_2\frac{k_1 - k_1^{-1}}{q-q^{-1}} \\
&= - (q+q^{-1}) e_1 e_2\frac{k_1  - k_1^{-1}}{q-q^{-1}}
- (q+q^{-1}) e_2  e_1 \frac{k_1 q - k_1^{-1}q^{-1}}{q-q^{-1}}\\
&
+(q+q^{-1}) e_2e_1 \frac{k_1 q- k_1^{-1}q^{-1}}{q-q^{-1}}  + (q+q^{-1})e_1e_2\frac{k_1 - k_1^{-1}}{q-q^{-1}}\\
&=0.
\end{aligned}
\]
We can also easily prove that $[f_2, S_{1 2}^{(+)}]=0$.

To prove $[f_i,  S_2^{(+)}]=0$ for all $i$, we note that $[f_1, S_2^{(+)}]=0$  by \eqref{eq:relat-5}.  Now
\[\begin{aligned}
&[f_2, S_2^{(+)}]=[f_2, e_2] e_2 +e_2 [f_2, e_2]
&=-\frac{k_2-k_2^{-1}}{q-q^{-1}}e_2 -  e_2 \frac{k_2-k_2^{-1}}{q-q^{-1}}.
\end{aligned}
\]
Since $k_2 e_2 k_2^{-1}=-e_2$ by \eqref{eq:relat-4}, we immediately see that the right hand side  is zero.

By using the following obvious automorphism of $\TU_q(\fsl_{2, 1})$,
\begin{eqnarray}\label{eq:auto}
\omega:  \widetilde{\U}_q(\fsl_{2, 1})\longrightarrow \widetilde{\U}_q(\fsl_{2, 1}),
\quad e_i\mapsto f_i, \  f_i\mapsto e_i, \ k_i \mapsto k_i^{-1}, \quad \forall i,
\end{eqnarray}
we obtain
\[
[e_i, S_{12}^{(-)}]=\omega([f_i, S_{12}^{(+)}]) = 0, \quad
[e_i, S_2^{(-)}]=\omega([f_i, S_2^{(+)}])=0, \quad \forall i.
\]
This proves the first part of the lemma.

(2).
The second part is also proven by direct computation.

Consider $\Delta(S_{12}^{(+)})$. Let us write
$[e_1, e_2]_{q^{-1}}=e_1 e_2 - q^{-1} e_2 e_1$, and consider its image under $\Delta$. 
We have 
\[
\begin{aligned}
\Delta(e_1 e_2) &= e_1 e_2\otimes k_1 k_2 + e_2\otimes e_1 k_2 +  q^{-1}e_1\otimes  e_2 k_1 + 1\otimes e_1 e_2, \\
\Delta(e_2 e_1) &= e_2 e_1\otimes k_1 k_2 + e_1\otimes e_2 k_1 +  q^{-1}e_2\otimes  e_1 k_2 + 1\otimes e_2 e_1.
\end{aligned}
\]
Combining these two equations, we obtain 
\begin{eqnarray}\label{eq:DeltaE13}
&\Delta([e_1, e_2]_{q^{-1}}) = [e_1, e_2]_{q^{-1}}\otimes k_1 k_2 + (1-q^{-2})e_2\otimes e_1 k_2 + 1\otimes [e_1, e_2]_{q^{-1}}.
\end{eqnarray}
Now we have 
\[
\begin{aligned}
\Delta(S_{12}^{(+)})&= e_1\otimes k_1\Delta([e_1, e_2]_{q^{-1}}) -  q\Delta([e_1, e_2]_{q^{-1}})e_1\otimes k_1\\
&\quad + 1\otimes e_1 \Delta([e_1, e_2]_{q^{-1}})  - q \Delta([e_1, e_2]_{q^{-1}}) 1\otimes e_1.
\end{aligned}
\]
Write the first line on the right hand side as $L1$, and the second line as $L2$. By using equation \eqref{eq:DeltaE13}, we easily obtain 
\[
\begin{aligned}
L1&= e_1[e_1, e_2]_{q^{-1}}\otimes k_1^2 k_2 + q^2(1-q^{-2})e_1e_2\otimes e_1 k_1 k_2 + q e_1\otimes [e_1, e_2]_{q^{-1}}k_1\\
&\quad - q ([e_1, e_2]_{q^{-1}}e_1\otimes k_1^2 k_2 + (1-q^{-2})e_2 e_1\otimes e_1 k_1 k_2 + e_1\otimes [e_1, e_2]_{q^{-1}} k_1)\\
&=S_{12}^{(+)}\otimes k_1^2 k_2 + (q^2-1)[e_1, e_2]_{q^{-1}}\otimes e_1 k_1 k_2,\\
L2&=[e_1, e_2]_{q^{-1}}\otimes e_1 k_1 k_2 + (1-q^{-2})e_2\otimes e_1 e_1 k_2 + 1\otimes e_1 [e_1, e_2]_{q^{-1}}\\
&\quad - q (q[e_1, e_2]_{q^{-1}}\otimes e_1 k_1 k_2 + q^{-1}(1-q^{-2})e_2\otimes e_1^2 k_2 + 1\otimes [e_1, e_2]_{q^{-1}}e_1)\\
&= (1-q^2)[e_1, e_2]_{q^{-1}}\otimes e_1 k_1 k_2 +  1\otimes S_{12}^{(+)},
\end{aligned}
\]
which immediately lead to 
\[
\begin{aligned}
\Delta(S_{12}^{(+)})=L1+L2
=S_{12}^{(+)}\otimes k_1^2 k_2+  1\otimes S_{12}^{(+)}.
\end{aligned}
\]
This proves the first relation in part (2).  The second relation can be proved in exactly the same way.

For the third relation in part (2), we have
\[
\begin{aligned}
\Delta(S_2^{(+)})&= (e_2\otimes k_2+1\otimes e_2)^2\\
&=S_2^{(+)}\otimes k_2^2+1\otimes S_2^{(+)}
+ e_2\otimes k_2 e_2 + e_2\otimes e_2 k_2,
\end{aligned}
\]
where the last two terms cancel upon using $k_2 e_2 = -e_2 k_2$, leading to
\[
\Delta(S_2^{(+)})=S_2^{(+)}\otimes k_2^2+1\otimes S_2^{(+)}.
\]
The last relation in part (2) can be proved in the same way.

The third part of the lemma is clear.
\end{proof}

\begin{theorem}\label{thm:sl21} Let $J$ be the two-sided ideal in $\TU_q(\fsl_{2, 1})$ generated by the elements
$ S_{1 2}^{(+)},  S_{1 2}^{(-)},   S_2^{(+)}$ and $ S_2^{(-)}$.  Then $J$ is a Hopf ideal, and
$\U_q(\fsl_{2, 1})=\TU_q(\fsl_{2, 1})/J$ is a Hopf algebra  with the induced co-multiplication, co-unit and antipode.
\end{theorem}
\begin{proof}
This immediately follows Lemma \ref{lem:Serre-relations} .
\end{proof}

%
%
%
%
\begin{remark}
We can boost $\U_q(\fsl_{2, 1})$ to a degenerate quantum general linear group $\U_q(\gl_{2, 1})$ in the obvious way.
 \end{remark}

The representation theory of $\U_q(\fsl_{2, 1})$ will be thoroughly treated in Section \ref{sect:reps-sl21}.

\section{Degenerate quantum general linear groups}\label{sect:typeA}

 In this section, we generalise the definition of the degenerate quantum special linear group $\U_q(\fsl_{2, 1})$ to arbitrary ranks.  To do this, it is easier to construct the corresponding degenerate quantum general linear group instead.

\subsection{The degenerate quantum general linear group $\U_q(\gl_{m, n})$}

Given a pair of positive integers $m, n$, we introduce the sets $I=\{1, 2, \dots, m+n\}$ and $I'=I\backslash\{m+n\}$.
Let $p=-q^{-1}$, and set $q_a=q $ if $a\le m$, and $q_a=p$ if $a>m$.

\begin{definition} \label{def:degen-gl} Let $\U_q(\gl_{m, n})$ be the unital associative algebra over $\C(q)$ generated by  the elements of  $\{e_a, f_a, K_b, K_b^{-1}\mid a\in I', \ b\in I\}$ subject to the following relations.
\begin{eqnarray}
&K_a K_a^{-1} =1,  \quad   K_a^{\pm 1} K_b^{\pm 1} =K_b^{\pm 1} K_a^{\pm 1}, \label{eq:gl-1}\\
&K_a e_b K_a^{-1} = q_a^{\delta_{a b}-\delta_{a, b+1}} e_b,  \label{eq:gl-2}\\
&K_a f_b K_a^{-1} = q_a^{-\delta_{a b}+\delta_{a, b+1}} f_b, \label{eq:gl-3}\\
&e_a f_b - f_b e_a = \delta_{a b} \frac{k_a - k_a^{-1}}{q_a-q_a^{-1}},  \quad  \text{with } \
k_a = K_aK_{a+1}^{-1},  \label{eq:gl-4}\\
&e_a e_b = e_b e_a, \quad f_a f_b = f_b f_a,  \quad |a-b|>1, \label{eq:gl-5}\\
& e_a^2 e_{a\pm 1} - (q_a+q_a^{-1}) e_a e_{a\pm 1} e_a +  e_{a\pm 1}e_a^2=0, \quad \ a\ne m, \label{eq:gl-6}\\
& f_a^2 f_{a\pm 1} - (q_a+q_a^{-1}) f_a f_{a\pm 1} f_a  +  f_{a\pm 1}f_a^2=0,
\quad \ a\ne m, \label{eq:gl-7}\\
& e_m^2=f_m^2=0,\label{eq:gl-8}\\
& e_m E_{m-1, m+2} - E_{m-1, m+2} e_m=0,\label{eq:gl-9}\\
&f_m E_{m+2, m-1} -  E_{m+2, m-1} f_m=0, \label{eq:gl-10}
\end{eqnarray}
where $E_{m-1, m+2}$ and $E_{m+2, m-1}$ are defined by
\[
\begin{aligned}
&E_{m-1, m+2}:=E_{m-1, m+1} e_{m+1} - q_{m+1}^{-1} e_{m+1} E_{m-1, m+1}, \\
&E_{m+2, m-1}:= f_{m+1} E_{m+1, m-1} - q_{m+1}E_{m+1, m-1} f_{m+1}, \\
&E_{m-1, m+1} := e_{m-1} e_m - q_m^{-1} e_m e_{m-1}, \\
&E_{m+1, m-1} := f_m f_{m-1}- q_m f_{m-1} f_m.\\
\end{aligned}
\]
Call $\U_q(\gl_{m, n})$ a degenerate quantum general linear group.
Denote by $\U_q(\fsl_{m, n})$ the subalgebra generated by the set $\{e_a, f_a, k_a, k_a^{-1}\mid a\in I'\}$, and call it a degenerate quantum special linear group.
\end{definition}

\begin{remark}
If $m=1$ or $n=1$, the relations \eqref{eq:gl-9} and \eqref{eq:gl-10} become vacuous.
\end{remark}
\begin{remark}
Note from the definition that
\[
k_m e_m k_m^{-1}= q_m q_{m+1}e_m = - e_m, \quad
k_m f_m k_m^{-1}= q_m^{-1} q_{m+1}^{-1}f_m = - f_m.
\]
In the case $m=2$ and $n=1$, these are the relations \eqref{eq:relat-4}.
\end{remark}

\subsection{Hopf algebraic structure of $\U_q(\gl_{m, n})$}

Now we examine the structure of the degenerate quantum general linear group $\U_q(\gl_{m, n})$.
\begin{lemma} \label{lem:hopf-mn}
Let $\widetilde{\U}_q(\gl_{m, n})$  be the unital algebra over $\C(q)$ generated by the set of elements $\{e_a, f_a, K_b, K_b^{-1}\mid a\in I', \ b\in I\}$ subject to the relations from \eqref{eq:gl-1} to \eqref{eq:gl-8}  only.  Then
$\widetilde{\U}_q(\gl_{m, n})$ has the structure of a Hopf algebra with co-multiplication
$\Delta: \widetilde{\U}_q(\gl_{m, n}) \longrightarrow \widetilde{\U}_q(\gl_{m, n})\otimes \widetilde{\U}_q(\gl_{m, n})$
\[
\begin{aligned}
\Delta(e_a)=e_a\otimes k_a + 1\otimes e_a, \  \
\Delta(f_a)=f_a\otimes 1 + k_a^{-1}\otimes f_a, \  \
\Delta(K_b)=K_b\otimes K_b,
\end{aligned}
\]
co-unit $\epsilon:  \widetilde{\U}_q(\gl_{m, n}) \longrightarrow \C(q)$
\[
\epsilon(e_a)=\epsilon(f_a)=0, \quad \epsilon(K_b)=1,
\]
and antipode $S: \widetilde{\U}_q(\gl_{m, n}) \longrightarrow \widetilde{\U}_q(\gl_{m, n})$
\[
S(e_a)=- e_a k_a^{-1}, \quad S(f_a)= - k_a f_a, \quad S(K_b)=K_b^{-1},
\]
for all  $a\in I'$ and $b\in I$.
\end{lemma}
\begin{proof}
Observe that the following facts readily imply the lemma.
\begin{itemize}
\item  the subalgebra generated by $\{K_b, K_b^{-1}, e_a, f_a\mid1\le  a<m, \ 1\le b\le m\}$ is isomorphic to the quantum group $\U_q(\gl_m)$ with the standard Hopf algebra structure;
\item the subalgebra generated by $\{K_b,  K_b^{-1}, e_a, f_a\mid 1\le a-m< n, \  1\le b-m\le n\}$ is isomorphic to the quantum group $\U_p(\gl_n)$ with the standard Hopf algebra structure;
\item the subalgebra generated by $\{k_a, e_a, f_a\mid a=m-1, m\}$ is isomorphic to $\U_q(\fsl_{2,1})$
with the Hopf algebra structure as given in Theorem \ref{thm:sl21}; and
\item the subalgebra generated by $\{k_a, e_a, f_a\mid a=m, m+1\}$  is isomorphic to $\U_p(\fsl_{2,1})$
with the Hopf algebra structure as given in Theorem \ref{thm:sl21}.
\end{itemize}
The new features not present in the context of ordinary Drinfeld-Jimbo quantum groups are in
the third and fourth dot points, which have already been dealt with in Section \ref{sect:sl3}.
\end{proof}

Let us define the following elements of $\widetilde{\U}_q(\gl_{m, n})$.
\[
\begin{aligned}
Q^+ := e_m E_{m-1, m+2} - E_{m-1, m+2} e_m, \quad
Q^-:=f_m E_{m+2, m-1} -  E_{m+2, m-1} f_m.
\end{aligned}
\]
\begin{lemma}\label{lem:Serre-Q} The elements $Q^{\pm}$ satisfy the following relations in  $\widetilde{\U}_q(\gl_{m, n})$.
\begin{eqnarray}
&f_a Q^+ -  Q^+ f_a =0, \quad  e_a Q^- - Q^- e_a=0 , \quad \forall a\in I';   \label{eq:Q-1}\\
&\Delta(Q^+)=Q^+\otimes k_{m-1} k_m^2 k_{m+1} + 1 \otimes Q^+,  \label{eq:Q-2}\\
&\Delta(Q^-)=Q^-\otimes 1 +  (k_{m-1} k_m^2 k_{m+1})^{-1} \otimes Q^-.  \label{eq:Q-3}
\end{eqnarray}
\end{lemma}
\begin{proof}
The proof of \eqref{eq:Q-1} requires some technical results which are listed in Appendix \ref{sect:commutaions}.  Note that the two relations in \eqref{eq:Q-1} imply each other by noting the analogue of the automorphism \eqref{eq:auto}. Thus we only need to prove one of them, and we will consider the first.

It is clear by \eqref{eq:gl-4} that $[f_a, Q^+]=0$ for all $a\ne m, m\pm1$.

For $[f_m, Q^+]$, we use \eqref{eq:lem3} to obtain
\begin{eqnarray}
\begin{aligned}
{[f_m, Q^+]}&=[f_m, e_m] E_{m-1, m+2} - E_{m-1, m+2}[f_m, e_m]\\
		&=- \frac{k_m- k_m^{-1}}{q_m-q_m^{-1}} E_{m-1, m+2}
		+ E_{m-1, m+2} \frac{k_m- k_m^{-1}}{q_m-q_m^{-1}} =0,
\end{aligned}
\end{eqnarray}
where in the last step, we used the fact that
\[
k_m E_{m-1, m+2} k_m^{-1} = -q_m^{-1}q_{m+1}^{-1} E_{m-1, m+2}
= E_{m-1, m+2}.
\]
Consider $[f_{m-1}, Q^+]$. Using \eqref{eq:lem4}, we have
\begin{eqnarray}
\begin{aligned}
-[f_{m-1}, Q^+]&=-e_m [f_{m-1}, E_{m-1, m+2}] + [f_{m-1}, E_{m-1, m+2}] e_m\\
		&= -e_m  E_{m, m+2} k_{m-1}^{-1} +  E_{m, m+2} k_{m-1}^{-1} e_m\\
		&= q_{m+1}^{-1} e_m e_{m+1} e_m k_{m-1}^{-1} + e_m k_{m-1}^{-1} e_{m+1}e_m\\
		&= (q_{m+1}^{-1} + q_{m-1})e_m e_{m+1} e_m k_{m-1}^{-1}=0,
\end{aligned}
\end{eqnarray}
where the last step follows from $q_{m+1}^{-1} + q_{m-1} = (-q^{-1})^{-1} + q=0$.

Finally,  we consider $[f_{m+1}, Q^+]$. We have
\begin{eqnarray}
\begin{aligned}
-[f_{m+1}, Q^+]&=e_m  E_{m-1, m+1} k_{m+1}q_{m+1}^{-1}
			-  E_{m-1, m+1} k_{m+1}q_{m+1}^{-1} e_m\\
			&=q_{m+1}^{-1}  (e_m  E_{m-1, m+1} - q_{m+1}^{-1} E_{m-1, m+1}  e_m)k_{m+1}=0.
\end{aligned}
\end{eqnarray}

The proof of \eqref{eq:Q-2} and \eqref{eq:Q-3} is straightforward but lengthy. To avoid interrupting the main line of thoughts, we relegate the details to Appendix \ref{sect:proof-Q}.
\end{proof}

The following result is an immediate corollary of the above lemma.
\begin{theorem}
Let $J$ be the two-sided ideal generated by $Q^+$ and $Q^-$ in $\widetilde{\U}_q(\gl_{m, n})$.  Then $J$ is a Hopf ideal, and
\[
\U_q(\gl_{m, n}) =\widetilde{\U}_q(\gl_{m, n})/J,
\]
which is a Hopf algebra with co-multiplication, co-unit and antipode induced by the corresponding structure maps of  $\widetilde{\U}_q(\gl_{m, n})$
given in Lemma \ref{lem:hopf-mn}.
\end{theorem}
\begin{proof} By \eqref{eq:Q-1}, $J$ is a proper two-sided ideal in $\widetilde{\U}_q(\gl_{m, n})$. It is a Hopf ideal by \eqref{eq:Q-2} and \eqref{eq:Q-3}, and the fact that $\epsilon(Q^\pm)=0$. Hence $\widetilde{\U}_q(\gl_{m, n})/J$ is a Hopf algebra with the structure maps induced by those of  $\widetilde{\U}_q(\gl_{m, n})$.  It is obvious that  $\U_q(\gl_{m, n})= \widetilde{\U}_q(\gl_{m, n})/J$.
\end{proof}

\subsection{More on the structure of $\U_q(\gl_{m, n})$}\label{sect:structures}
We now return to the Hopf algebra $\U_q(\gl_{m, n})$.
By examining the defining relations,  we can easily see
that each of the subsets of generators below generates
a subalgebra (with identity).
\[
\begin{aligned}
(a). &\qquad
\{e_a\mid a\in I'\}, \quad \{f_a\mid a\in I'\}, \quad \{K_b^{\pm 1}\mid b\in I\},\\
(b). &\qquad
\{e_a, K_b^{\pm 1}\mid a\in I', \ b\in I\}, \quad \{e_a, K_b^{\pm 1}, f_c \mid a\in I', \ b\in I, \ m\ne c\in I'\}
\end{aligned}
\]
We denote the subalgebras generated by the subsets in $(a)$
by $\U^+$, $\U^-$ and $\U^0$ respectively, and those by the subsets in $(b)$ by
$\U_q(\fb)$ and $\U_q(\fp)$ respectively. Note that $\U_q(\fb)$ and $\U_q(\fp)$ are Hopf subalgebras of $\U_q(\gl_{m, n})$.

We have the following easy observation.
\begin{lemma}\label{lem:triangular}
The algebra $\U_q(\gl_{m, n})$ admits the following triangular decomposition
$\U_q(\gl_{m, n})=\U^- \U^0 \U^+$, that is,  every element of $\U_q(\gl_{m, n})$ can be expressed as a linear combination of elements of the form $u_- u_0 u_+$ with
$u_-\in\U^-$, $u_0\in\U^0$ and $u_+\in\U^+$.  Furthermore, $\U_q(\fb)=\U^0\U^+$.
\end{lemma}
\begin{proof}
Given any product of the generators, we can always use the defining relations to move $e_a$'s to the right of $f_a$'s and $K_b^{\pm 1}$'s, and  move $K_b^{\pm 1}$'s to the right of $f_a$'s, thus to express the product as a linear combination of elements of the form described in the lemma. The statements in the lemma easily follow from this observation.
\end{proof}

To further analyse the structure of $\U_q(\gl_{m, n})$, we need some notation.
We adopt the standard notation of the $x$-commutator
\[
[A, B]_x = A B - x B A
\]
for any $A, B\in\U_q(\gl_{m, n})$ and $x\in\C(q)$.
The usual commutator is recovered from $[A, B]=[A, B]_1$.
We will also use the quantum adjoint action defined by
\begin{eqnarray}\label{eq:ad}
\begin{array}{l}
ad: \U_q(\gl_{m, n})\otimes \U_q(\gl_{m, n}) \longrightarrow \U_q(\gl_{m, n}), \\
x\otimes y\mapsto
ad_x(y):=\sum_{(x)} x_{(1)} y S(x_{(2)}).
\end{array}
\end{eqnarray}
Note that if $A\in \U_q(\gl_{m, n})$ satisfies $K_c A K_c^{-1}= q_c^{\lambda_c}A$ for all $c\in I$, where  $\lambda_c$ are some integers,  then
$
ad_{f_c}(A)=f_c A - q_c^{-\lambda_c}q_{c+1}^{\lambda_{c+1}}A f_c = [f_c, A]_{q_c^{-\lambda_c}q_{c+1}^{\lambda_{c+1}}}.
$

Now for any $1\leq i<j\leq m+n,$ we define inductively the elements $E_{ij}$, $E_{ji}$ of $\U_q(\fgl_{m, n})$ by $E_{i,i+1}=e_i$,  \ $E_{i+1,i}=f_i$, and
\begin{eqnarray}
&E_{ij}=E_{i,j-1}E_{j-1,j}-q^{-1}_{j-1}E_{j-1,j}E_{i,j-1}, \quad \text{for $j>i+1$},  \label{def-Eij}\\
&E_{ji}=E_{j,j-1}E_{j-1,i}-q_{j-1}E_{j-1,i}E_{j,j-1}, \quad \text{for $j>i+1$}. \label{def-Eji}
\end{eqnarray}

We will concentrate on the $E_{j i}$ belonging to $\U^-$.
By using the fact that $f_a$ and $f_b$  commute if $|a-b|>1$, one can easily show that for any $k$ satisfying $j>k>i$,
\begin{eqnarray} \label{eq:iteration}
E_{j i}=E_{j k} E_{ki}- q_k E_{ki} E_{jk}=[E_{j k}, E_{ki}]_{q_k}.
\end{eqnarray}
We further observe that
\[
\begin{aligned}
&{[f_i, f_{i-1}]}_{q_i}= f_i f_{i-1}- q_i f_{i-1}  f_i= -q_i ad_{f_{i-1}}(f_i), \\
&E_{j i} = (-1)^{j-i-1}\prod_{k=i+1}^{j-1} q_k \cdot ad_{f_i} ad_{f_{i+1}} \dots ad_{f_{j-2}}(f_{j-1}), \quad j>i,
\end{aligned}
\]
where we used  \eqref{eq:iteration} to obtain the second relation.

\begin{lemma}\label{lem:Ejifk}
Assume that $i<j$. If $\{i,j\}\ne \{k, k+1\}$, then
\begin{eqnarray}\label{eq:Ejifk}
  [E_{ji},E_{k+1,k}]=0.
\end{eqnarray}
\begin{proof}
There are three possibilities with $i<j<k$,  $k+1<i<j$, and  $i<k<k+1<j$ respectively.
Equation \eqref{eq:Ejifk} obviously holds in the first two cases since $f_a$ and $f_b$ commute if $|a-b|>1$.  In the last case, we can express $E_{ji}$ as
\[
E_{j i} = [E_{j, k+2},  [E_{k+2, k-1}, E_{k-1, i}]_{q_{k-1}}]_{q_{k+2}}.
\]
[If $j=k+2$ or $i=k-1$, we only need one $q$-commutator.]
By the two cases already proved,  we immediately see that
\[
[E_{j i}, E_{k+1,k}] = [E_{j, k+2},  [[E_{k+2, k-1}, E_{k+1,k}],  E_{k-1, i}]_{q_{k-1}}]_{q_{k+2}}.
\]
Thus \eqref{eq:Ejifk} holds if \begin{eqnarray}\label{eq:reducedEE}
[E_{k+2,k-1},f_k]=0.
\end{eqnarray}

If $k=m$, this is nothing else but the quartic Serre relation \eqref{eq:gl-10}.

If $k\ne m$, one way to prove \eqref{eq:reducedEE} is by simply expanding $[E_{k+2,k-1},f_k]$ in terms of $f_{k\pm 1}$ and $f_k$, then applying the cubic Serre relation \eqref{eq:gl-7}.   A conceptually clearer way is to use the quantum adjoint action. We have
\[
\begin{aligned}
E_{k+2, k-1} &= q_{k}q_{k+1} ad_{f_{k-1}}ad_{f_k}(f_{k+1}), \\
{[E_{k+2, k-1}, f_k]} &= -q_{k}q_{k+1}ad_{f_k}ad_{f_{k-1}}ad_{f_k}(f_{k+1}) \\
			&= -q_{k}q_{k+1} ad_{f_k f_{k-1} f_k}(f_{k+1}).
\end{aligned}
\]
Using the cubic Serre relation \eqref{eq:gl-7}, we obtain
\[
\begin{aligned}
ad_{f_k f_{k-1} f_k}(f_{k+1})&= \frac{1}{q_{k}+q_{k}^{-1}} ad_{f_k^2 f_{k-1} + f_{k-1} f_k^2}(f_{k+1})\\
&= \frac{1}{q_{k}+q_{k}^{-1}} \left(ad_{f_k^2}ad_{f_{k-1}}+ad_{f_{k-1}} ad_{f_k^2}\right)(f_{k+1}).\\	
\end{aligned}
\]
By \eqref{eq:gl-5} and \eqref{eq:gl-7}, we have
\begin{eqnarray}\label{eq:ad-Serre}
\begin{aligned}
ad_{f_{k-1}}(f_{k+1})=0,  \quad
ad_{f_k^2}(f_{k+1})=0.
\end{aligned}
\end{eqnarray}
Hence $ad_{f_k f_{k-1} f_k}(f_{k+1})=0$, which immediately leads to $[E_{k+2, k-1}, f_k]=0$. This completes the proof.
\end{proof}
\end{lemma}

\begin{lemma}\label{lem:Ekikj}
The elements $E_{j i}$ ($1\le i<j\le m+n$) satisfy the following relations.
\begin{eqnarray}
&E_{ki}^2=0,\quad  i\leq m<k, \label{eq:nil}\\
&{[E_{ji},E_{\ell k}]}=0, \quad i<j<k<\ell \text{ or }  k<i<j<\ell, \label{eq:EjiElk1} \\
&E_{ki}E_{kj}=q_kE_{kj}E_{ki}, \quad i<j<k, \label{eq:EkiEkj}\\
&E_{jk}E_{ik}=q_k^{-1}E_{ik}E_{jk}, \quad k<i<j,\label{eq:EjkEik} \\
&{[E_{ji},E_{\ell k}]}=(q-q^{-1})E_{\ell i}E_{jk}, \quad  i<k<j<\ell. \label{eq:EjiElk2}
\end{eqnarray}
\end{lemma}

\begin{remark} \label{rem:order}
Given any monomial in the elements $E_{j i}$ ($i<j$),  Lemma \ref{lem:Ekikj} enables us to express it as a linear combination of ordered monomials for any chosen linear order of the elements.
\end{remark}

\begin{proof}[Proof of Lemma \ref{lem:Ekikj}]
First we consider equation \eqref{eq:EjiElk1}.
From \eqref{def-Eji} we can see that $E_{j i}$ can be expressed as a linear combination of  products of the form $f_{a_1}f_{a_2}\dots f_{a_{j-i}}$ with the $a_r$ distinct elements of $\{i, i+1, \dots, j-1\}$. Hence $[E_{j i}, E_{\ell k}]$ can be expressed as a linear combination of the elements
\begin{eqnarray}\label{eq:bracket-EE}
[ f_{a_1}f_{a_2}\dots f_{a_{j-i}}, E_{\ell k}]= \sum_{r=1}^{j-i} f_{a_1}\dots f_{a_{r-1}} [f_{a_r},  E_{\ell k}] f_{a_{r+1}}\dots f_{a_{j-i}}.
\end{eqnarray}
Note that we have $\{a_r, a_r+1\}\cap\{k, \ell\}=\emptyset$ for all $r$. Hence all terms on the right side of \eqref{eq:bracket-EE} vanish by \eqref{eq:Ejifk}. This proves \eqref{eq:EjiElk1}.

Next we prove \eqref{eq:EkiEkj}.

Consider the special case with $j=k-1$ and $i=k-2$. We have
\[
\begin{aligned}
E_{k, k-2} E_{k,k-1} = E_{k,k-1}  E_{k-1, k-2} E_{k,k-1} - q_{k-1} E_{k-1, k-2} E_{k,k-1}^2, \\
 E_{k,k-1}E_{k, k-2} = E_{k,k-1}^2  E_{k-1, k-2}  - q_{k-1} E_{k,k-1}  E_{k-1, k-2} E_{k,k-1}.
\end{aligned}
\]
If $k\ne m+1$, by using \eqref{eq:gl-7}, we obtain
\[
\begin{aligned}
 E_{k,k-1}E_{k, k-2} &= E_{k,k-1}^2  E_{k-1, k-2}  - q_{k-1} E_{k,k-1}  E_{k-1, k-2} E_{k,k-1}\\
&=  q_{k-1}^{-1} E_{k,k-1}  E_{k-1, k-2} E_{k,k-1} - E_{k-1, k-2} E_{k,k-1}^2 \\
&= q_{k-1}^{-1} (E_{k,k-1}  E_{k-1, k-2} E_{k,k-1} - q_{k-1}E_{k-1, k-2} E_{k,k-1}^2)\\
&= q_{k-1}^{-1}E_{k, k-2} E_{k, k-1}.
\end{aligned}
\]
Note that in this case,  $q_k=q_{k-1}$, and we arrive at \eqref{eq:EkiEkj}.
If $k=m+1$, we have
\[
\begin{aligned}
E_{k, k-2} E_{k,k-1} = E_{k,k-1}  E_{k-1, k-2} E_{k,k-1}, \\
 E_{k,k-1}E_{k, k-2} =- q_{k-1} E_{k,k-1}  E_{k-1, k-2} E_{k,k-1},
\end{aligned}
\]
and $q_{k-1} = -q_k^{-1}$. Hence follows \eqref{eq:EkiEkj}  in this case.
Therefore for any $k$,
\begin{eqnarray}\label{eq:Ekikj4}
E_{k,k-2}E_{k,k-1}=q_kE_{k,k-1}E_{k,k-2}.
\end{eqnarray}

Now we consider the general case of \eqref{eq:EkiEkj} with $k-1>j>i<k-2$.  We have
\begin{eqnarray}\label{eq:Ekikj2}
\begin{aligned}
E_{ki}E_{kj}=E_{ki}(E_{k,k-1}E_{k-1,j}-q_{k-1}E_{k-1,j}E_{k,k-1})\\ \overset{(\ref{eq:EjiElk1})}{=}E_{ki}E_{k,k-1}E_{k-1,j}-q_{k-1}
E_{k-1,j}E_{ki}E_{k,k-1}.
\end{aligned}
\end{eqnarray}
As $i<k-2$, we have
\begin{eqnarray}\label{eq:Ekikj3}
\begin{aligned}
E_{ki}E_{k,k-1}=&(E_{k,k-2}E_{k-2,i}-q_{k-2}E_{k-2,i}E_{k,k-2})E_{k,k-1}\\ \overset{(\ref{eq:EjiElk1})}{=}&E_{k,k-2}E_{k,k-1}E_{k-2,i}-q_{k-2}
E_{k-2,i}E_{k,k-2}E_{k,k-1}.
\end{aligned}
\end{eqnarray}
Using \eqref{eq:Ekikj4} to the right hand side of \eqref{eq:Ekikj3}, we obtain
\begin{eqnarray}\label{eq:Ekikj5}
\begin{aligned}
E_{ki}E_{k,k-1}=&q_k(E_{k,k-1}E_{k,k-2}E_{k-2,i}-q_{k-2}E_{k-2,i}E_{k,k-1}E_{k,k-2})\\
=&q_kE_{k,k-1}(E_{k,k-2}E_{k-2,i}-q_{k-2}E_{k-2,i}E_{k,k-2})\\
=&q_kE_{k,k-1}E_{ki}.
\end{aligned}
\end{eqnarray}
Using \eqref{eq:Ekikj5} in \eqref{eq:Ekikj2}, we obtain
\begin{align*}
 E_{ki}E_{kj}=&q_k(E_{k,k-1}E_{ki}E_{k-1,j}-q_{k-1}E_{k-1,j}E_{k,k-1}E_{ki})\\
=&q_k(E_{k,k-1}E_{k-1,j}-q_{k-1}E_{k-1,j}E_{k,k-1})E_{ki}\\
=&q_kE_{kj}E_{ki}.
\end{align*}
This completes the proof of \eqref{eq:EkiEkj}.
Equation \eqref{eq:EjkEik} can be proved similarly.

We then consider the equation \eqref{eq:EjiElk2}. We have
\begin{align*}
 [E_{ji},E_{lk}]=&E_{ji}(E_{lj}E_{jk}-q_{j}E_{jk}E_{lj})-(E_{lj}E_{jk}-q_{j}E_{jk}E_{lj})E_{ji}\\
\overset{(\ref{eq:EkiEkj})}{=}&(E_{ji}E_{lj}E_{jk}-q_{j}^{-1}E_{lj}E_{ji}E_{jk})-q_j(E_{ji}E_{jk}E_{lj}-
E_{jk}E_{lj}E_{ji})\\
\overset{(\ref{eq:EkiEkj})}{=}&-q^{-1}_jE_{li}E_{jk}+q_jE_{jk}E_{li}\\
\overset{(\ref{eq:EjiElk1})}{=}&(q_j-q^{-1}_j)E_{li}E_{jk}.
\end{align*}
This proves \eqref{eq:EjiElk2}.

Finally, we prove \eqref{eq:nil}. We have, for $k>m\ge i$,
\[\begin{aligned}
E_{ki}^2=&(E_{km}E_{mi}-q_mE_{mi}E_{km})E_{ki}\\
=&q_iq_k^{-1}E_{ki}E_{km}E_{mi}-q_mq_k^{-1}q_iE_{ki}E_{mi}E_{km}\\
=&q_iq_k^{-1}E_{ki}^2
=-q^2E_{ki}^2.
\end{aligned}\]
This immediately leads to $E_{ki}^2=0$.
\end{proof}

Denote by $\U_q(\fu_-)$ the subalgebra of $\U_q(\gl_{m, n})$ generated by the elements  $E_{m+j, i}$ with $i=1, 2,\dots, m$ and $j=1, 2, \dots n$.
For $\theta_{i,n}\cdots\theta_{i,1}\in\{0,1\}$, we let
\[
\begin{aligned}
\Gamma_i^{(\theta_i)}&=(E_{m+n,i})^{\theta_{i, n}}(E_{m+n-1,i})^{\theta_{i, n-1}}\cdots(E_{m+1,i})^{\theta_{i 1}},\\
\Gamma^{(\theta)}&=\Gamma_1^{({\theta_1})}\Gamma_2^{({\theta_2})}\cdots\Gamma_m^{({\theta_m})}.
\end{aligned}
\]
Then it follows from Lemma \ref{lem:Ejifk} and Remark \ref{rem:order} that the $\Gamma^{(\theta)}$ with $\theta\in\{0,1\}^{\times mn}$ span $\U_q(\fu_-)$.

Consider the subalgebra $\U_q(\fg_0):=\U_q(\gl_m)\otimes\U_p(\gl_n)$ in $\U_q(\gl_{m, n})$. Let us write
$\U^-_{m, n}= \U^- \cap \U_q(\fg_0)$.  Then Lemma \ref{lem:Ekikj} enables us to express $\U^-$ as $\U^-= \U_q(\fu_-)\U^-_{m, n}$.  By Lemma \ref{lem:triangular},  $\U_q(\gl_{m, n})= \U_q(\fu_-)\U^-_{m, n} \U^0\U^+ = \U_q(\fu_-)\U_q(\fp)$.

To summarise, we have the following result.
\begin{proposition}\label{prop:para} The degenerate quantum general linear group has the  parabolic decomposition
$
\U_q(\gl_{m,n})=\U_q(\fu_-)\U_q(\fp),
$
where $\U_q(\fu_-)$ is spanned by the elements $\Gamma^{(\theta)}$ with $\theta\in\{0,1\}^{\times mn}$.
\end{proposition}

\section{Finite dimensional representations of $\U_q(\gl_{m, n})$}
\subsection{Finite dimensional irreducible representations}\label{sect:rep-theory}

The triangular decomposition (Lemma \ref{lem:triangular}) and parabolic decomposition (see  Proposition \ref{prop:para}) of $\U_q(\gl_{m, n})$ enable us to adapt the usual construction of highest weight modules in Lie theory to the present context. We develop the construction here, obtaining a systematic method for studying the representation theory of the degenerate quantum general linear group. 

Consider representations of the subalgebra $\U_q(\fg_0)$ of $\U_q(\gl_{m, n})$.
A highest weight $\U_q(\fg_0)$-module is one generated by a highest weight vector
$v$ such that
\[
\begin{aligned}
e_i v=0, \quad  K_j v = \lambda_j v, \quad 1\le i\le  m-1, \ 1\le j\le  m, \\
e_{m+\mu} v=0, \quad  K_{m+\nu} v= \lambda_{m+\nu} v, \quad 1\le \mu\le n-1, \ 1\le\nu\le n,
\end{aligned}
\]
where  $0\ne\lambda_a\in\C(q)$ for all $a$.  Write $\Lambda=(\lambda_1, \lambda_2, \dots, \lambda_{m+n})$ and call it the highest weight of the module.  It is known that every finite dimensional simple weight module for $\U_q(\fg_0)$ is a highest weight module; a simple highest weight module is finite dimensional if and only if
\begin{eqnarray}\label{eq:fd}
\frac{\lambda_a}{\lambda_{a+1}}= \omega_a q_a^{\ell_a}, \quad \text{where $\ell_a\in\Z_+$, $\omega_a=\pm 1$, for all $a\ne m$}.
\end{eqnarray}

Given any simple $\U_q(\fg_0)$-module $L^0(\Lambda)$ with  highest weight $\Lambda$,
we boost it to a $\U_q(\fp)$-module by requiring $e_mL^0(\Lambda)=\{0\}$.  We can then construct the generalised Verma module
$V(\Lambda) := \U_q(\gl_{m, n})\otimes_{\U_q(\fp)}L^0(\Lambda)$, which as a vector space is given by
\[
V(\Lambda) = \U_q(\fu_-)\otimes L^0(\Lambda).
\]

The generalised Verma module contains a  unique maximal submodule $M(\Lambda)$, which is  the sum of all the submodules having zero intersection with  $1\otimes L^0(\Lambda)$. Thus $V(\Lambda)$ has a unique simple quotient $L(\Lambda):=V(\Lambda)/M(\Lambda)$.

We have the following result.
\begin{theorem} \label{thm:fd-irreps} The simple  $\U_q(\gl_{m, n})$-module  $L(\Lambda)$  is finite dimensional if and only if its highest weight $\Lambda$ satisfies the condition \eqref{eq:fd}.
\end{theorem}
\begin{proof}
It is clear that \eqref{eq:fd} is a necessary condition for $L(\Lambda)$ to be finite dimensional. It is also sufficient since $\U_q(\fu_-)$ is finite dimensional.
\end{proof}

\begin{remark} \label{rem:sim-super}
The parametrisation of the finite dimensional simple  $\U_q(\gl_{m, n})$-modules is essentially the same as that for the quantum general linear supergroup $\U_q(\gl_{m|n})$, see \cite{Z93, Z98}.
\end{remark}

\subsection{Tensor representations}\label{sect:tensors}
The quantum general linear group $\U_q(\gl_{m, n})$ admits a class of finite dimensional representations analogous to the tensor representations of the Drinfeld-Jimbo quantum general linear group. We study these representations here. 

%
%
%
Let $V=\C(q)^{m+n}$, and fix the standard basis
\[
v_1=\begin{pmatrix}1\\ 0\\ 0\\ \vdots \\ 0\\ 0\end{pmatrix}, \quad
v_2=\begin{pmatrix}0\\ 1\\ 0\\ \vdots \\ 0\\ 0\end{pmatrix}, \quad
\dots, \quad
v_{m+n}=\begin{pmatrix}0\\ 0\\ 0\\ \vdots \\ 0\\ 1\end{pmatrix}.
\]
Let $e_{a b}$ ($a,b\in I$) be the matrix units of size $(m+n)\times(m+n)$ defined relative to this basis. Then $e_{a b}v_c = \delta_{b c} v_a$  for all $a, b, c\in I$.
Denote by $V^*$ the dual space of $V$, and let
\[
\begin{aligned}
&\bar{v}_1=\begin{pmatrix}1& 0& 0& \dots & 0& 0\end{pmatrix}, \\
&\bar{v}_2=\begin{pmatrix}0& 1& 0& \dots & 0& 0\end{pmatrix}, \\
&\qquad \dots\dots\dots\dots\\
&\bar{v}_{m+n}=\begin{pmatrix}0& 0& 0& \dots & 0& 1\end{pmatrix},
\end{aligned}
\]
which form a basis of $V^*$ dual to the standard basis of $V$ in the sense that
$\bar{v}_a(v_b)=\delta_{a b}$ for all $a, b\in I$.
We can endow $V$ with a $\U_q(\gl_{m, n})$-module structure as follows.
\begin{lemma}
There is a  $\U_q(\gl_{m, n})$-action on $V$ defined,  for all  $a\in I'$,  $b, c\in I$,  by
\begin{eqnarray}\label{eq:act}
e_a v_c = \delta_{a+1, c} v_a, \quad f_a v_c = \delta_{a c} v_{a+1}, \quad
K_b^{\pm 1} v_c = q_b^{\pm \delta_{b c}} v_c.
\end{eqnarray}
The corresponding  representation $\nu: \U_q(\gl_{m, n})\longrightarrow \End_{\C(q)}(V)$ is given by
\begin{eqnarray}\label{eq:rep}
\nu(e_a)=e_{a, a+1}, \quad \nu(f_a)=e_{a+1, a}, \quad  \nu(K_b)=1 +(q_b-1)e_{b b}.
\end{eqnarray}
\end{lemma}
\begin{proof} The second part of the lemma is a simple consequence of the first,  thus we only need to prove that \eqref{eq:act} defines a $\U_q(\gl_{m, n})$-module.

It is clear that \eqref{eq:act}  respects the relation \eqref{eq:gl-1},  and also the relations \eqref{eq:gl-5}--\eqref{eq:gl-10} since $e_a^2$ and $f_a^2$ for all $a\in I'$ act on $V$ by zero. Thus we only need to check the relations  \eqref{eq:gl-2}, \eqref{eq:gl-3} and \eqref{eq:gl-4}.

Let us consider  \eqref{eq:gl-4} first. The case $a\ne b$ is clear.  In the case $a=b$, we note that
\[
\begin{aligned}
\frac{k_a - k_a^{-1}}{q_a - q_a^{-1}} v_c &= \frac{ q_a^{\delta_{a c}}q_{a+1}^{-\delta_{a+1, c}} -
q_a^{-\delta_{a c}}q_{a+1}^{\delta_{a+1, c}}}{q_a - q_a^{-1}} v_c.
\end{aligned}
\]
Using the first one of the following relations
\begin{eqnarray}\label{eq:exp}
q_c^{\pm \delta_{a b}}= 1 + (q_c^{\pm 1} - 1)\delta_{a b}, \quad q_a^{\pm \delta_{a b}} = q_b^{\pm \delta_{a b}},  \quad \forall a, b, c\in I,
\end{eqnarray}
and the fact that $p-p^{-1}=q-q^{-1}$, we can rewrite the scalar factor in front of $v_c$ on the right hand side as
\[
(\delta_{a c}-\delta_{a+1, c})\frac{q_c - q_c^{-1}}{q_a - q_a^{-1}} = \delta_{a c}-\delta_{a+1, c}.
\]
We can easily work out the action of $e_a f_a - f_a e_a$ on  $v_c$ for any $c$, and we obtain
\[
\begin{aligned}
(e_a f_a - f_a e_a)v_c &= (\delta_{a c}  - \delta_{a+1, c}) v_c=
\frac{k_a - k_a^{-1}}{q_a - q_a^{-1}} v_c.
\end{aligned}
\]

For \eqref{eq:gl-2} and \eqref{eq:gl-3}, we have
\[
\begin{aligned}
K_a e_b K_a^{-1} v_c
&= \delta_{b+1, c} q_a^{-\delta_{a c}+\delta_{a b}}v_b
= q_a^{\delta_{a b}-\delta_{a, b+1}} e_b v_c, \\
K_a f_b K_a^{-1} v_c
&= \delta_{b c} q_a^{-\delta_{a c}+\delta_{a, b+1}} v_{b+1}
= q_a^{-\delta_{a b} +\delta_{a, b+1}} f_b v_c.
\end{aligned}
\]

This completes the proof.
\end{proof}

Since $\U_q(\gl_{m, n})$ is a Hopf algebra, the dual space $M^*$ of any finite dimensional $\U_q(\gl_{m, n})$-module $M$ is naturally a $\U_q(\gl_{m, n})$-module.  For any  $x\in\U_q(\gl_{m, n})$ and $\bar{v}\in M^*$,  we define $x\bar{v}$ by
\[
(x\bar{v})(w) = \bar{v}(S(x)w), \quad \forall w\in M.
\]
In particular, the dual space $V^*$ of $V$ is a $\U_q(\gl_{m, n})$-module with
\begin{eqnarray}
K_b \bar{v}_c = q_b^{-\delta_{b c}} \bar{v}_c, \quad
e_a \bar{v}_c = - \delta_{a c} q_{a+1} \bar{v}_{a+1},  \quad
f_a \bar{v}_c = - \delta_{a+1, c} q_{a+1}^{-1} v_a
\end{eqnarray}
for all $a\in I'$ and $b, c\in I$.

\begin{remark}
The  highest weight vector of $V$ is $v_1$ with weight $(q, 1, \dots, 1)$, and the highest weight vector of $V^*$ is $\bar{v}_{m+n}$ with weight $(1, \dots, 1, p^{-1})$, where we recall that $p=-q^{-1}$.
\end{remark}

%
%
%
%
By using the Hopf algebraic structure of $\U_q(\gl_{m, n})$, we can turn the tensor product of any $\U_q(\gl_{m, n})$-modules into a $\U_q(\gl_{m, n})$-module.
In particular,  we have the $\U_q(\gl_{m, n})$-modules $V^{\otimes r}\otimes (V^*)^{\otimes s}$ for $r, s=1, 2, \dots$.  We will call them tensor modules.
Note that $\U_q(\gl_{m, n})$ acts on these modules through the iterated co-multiplication
\begin{eqnarray}\label{eq:Delta-r}
\Delta^{(r+s-1)}= (\Delta\otimes\id^{\otimes (r+s-2)}) (\Delta\otimes\id^{\otimes (r+s-3)})\dots (\Delta\otimes\id)\Delta.
\end{eqnarray}

\begin{example}\label{ex:VV} The tensor square $V\otimes V$ of the natural $\U_q(\gl_{m, n})$-module $V$ decomposes into the direct sum of two simple modules $L_s=L(q^2, 1, \dots, 1)$ and $L_a=L(q, q, 1, \dots, 1)$, which are respectively generated by the highest weight vectors
$v_1\otimes v_1$ and $v_1\otimes v_2 - q^{-1} v_2\otimes v_1$.

\noindent
A basis for $L_s$:
\[
\begin{aligned}
&\{v_i\otimes v_i, \ v_j\otimes v_k + q^{-1} v_k\otimes v_j\mid i, j, k=1, \dots, m, \ j<k \} \\
&\cup\{ v_i\otimes v_{m+\mu} -p v_{m+\mu}\otimes v_i\mid 1\le i\le m, \ 1\le \mu\le n\}\\
&\cup\{v_{m+\mu}\otimes v_{m+\nu} - p  v_{m+\nu}\otimes v_{m+\mu} \ \mid 1\le \mu<\nu\le n\};\end{aligned}
\]
A basis for $L_a$:
\[
\begin{aligned}
&\{\ v_i\otimes v_j -  q^{-1} v_j\otimes v_i\mid 1\le i<j\le m\} \\
&\cup\{ v_i\otimes v_{m+\mu} + p v_{m+\mu}\otimes v_i\mid 1\le i\le m, \ 1\le \mu\le n\}\\
&\cup\{v_{m+\mu}\otimes v_{m+\mu}, \ v_{m+\mu}\otimes v_{m+\nu} + p  v_{m+\nu}\otimes v_{m+\mu} \ \mid \mu, \nu =1, \dots, n, \ \mu<\nu\}.
\end{aligned}
\]
\end{example}

%
%
%
%
\begin{remark}
The degenerate quantum general linear group  $\U_q(\gl_{m, n})$ is also a Hopf algebra with
the opposite co-multiplication given by
\begin{eqnarray}\label{eq:op-comult}
\begin{aligned}
&\Delta': \U_q(\gl_{m, n}) \longrightarrow \U_q(\gl_{m, n})\otimes \U_q(\gl_{m, n}), \\
&\Delta'(e_a)=e_a\otimes 1+ k_a \otimes e_a, \\
&\Delta'(f_a)=f_a\otimes k_a^{-1}+ 1\otimes f_a, \\
&\Delta'(K_b)=K_b\otimes K_b.
\end{aligned}
\end{eqnarray}

Given any two $\U_q(\gl_{m, n})$-modules, we may then endow their tensor product with a module structure by using the opposite co-multiplication $\Delta'$. An immediate question is whether the module structures with respect to the co-multiplication and the  opposite co-multiplication are isomorphic.
\end{remark}
\subsection{A solution of the Yang-Baxter equation}\label{sect:R-matrix}
We answer the above question in the affirmative for the modules $V^{\otimes r}$ for all $r$. This requires the construction of an $R$-matrix.

Introduce the element $R\in \End_{\C(q)}(V\otimes V)$ such that
\begin{eqnarray} \label{eq:R}
R&=R_0 \Theta,
\end{eqnarray}
where $R_0$ and $\Theta$ are respectively defined by
\[
\begin{aligned}
R_0:= 1\otimes 1 + \sum_{a\in I} (q_a-1)e_{a a}\otimes e_{a a}, \  \
\Theta:=1\otimes 1 +  (q-q^{-1})\sum_{a<b}e_{a b}\otimes e_{b a}.
\end{aligned}
\]
We can easily see that
\[
R(v_a\otimes v_b) = \left\{
\begin{array}{l l}
v_a\otimes v_b, \quad \text{if $a<b$}, \\
q_a v_a\otimes v_a, \quad \text{if $a=b$}, \\
v_a\otimes v_b+ (q-q^{-1}) v_b\otimes v_a , \quad \text{if $a>b$}.
\end{array}
\right.
\]

\begin{remark}
Note that this $R$-matrix differs quite significantly from the $R$-matrix in the natural representation of
the usual quantum general linear group (see \eqref{eq:standard-R}) and that of
the quantum general linear supergroup (see \cite[p533]{Z98}).
\end{remark}

We have the following result.
\begin{lemma}\label{lem:R-matrix}
The matrix $R$ defined by \eqref{eq:R} has the following properties.
\begin{enumerate}
\item $R$ is invertible and satisfies the Yang-Baxter equation
\begin{eqnarray}\label{eq:YBE-R}
R_{12} R_{13}R_{23} = R_{23} R_{13} R_{12}.
\end{eqnarray}
\item For all $x\in \U_q(\gl_{m, n})$,
\begin{eqnarray}\label{eq:endo}
R(\nu\otimes\nu)\Delta(x) = (\nu\otimes\nu)\Delta'(x)R.
\end{eqnarray}
\end{enumerate}
\end{lemma}
\begin{proof} (1). It is clear that
\[
\begin{aligned}
R_0^{-1}&= 1\otimes 1 + \sum_{a\in I} (q_a^{-1}-1)e_{a a}\otimes e_{a a}, \quad
\Theta^{-1}&=1\otimes 1 -  (q-q^{-1})\sum_{a<b}e_{a b}\otimes e_{b a}.
\end{aligned}
\]
Hence $R^{-1}=\Theta^{-1} R_0^{-1}$.

To prove that $R$ satisfies the Yang-Baxter equation, it is useful to recall the standard $R$-matrix  in the natural representation of $\U_q(\gl_{m+n})$. We  denote it by $T$, which can be expressed as
\begin{eqnarray}\label{eq:standard-R}
\begin{aligned}
T&=T_0\Xi, \\
T_0&= 1\otimes 1 + (q-1)\sum_{a\in I} e_{a a}\otimes e_{a a}, \\
\Xi&=1\otimes 1 + (q-q^{-1})\sum_{a<b}e_{a b}\otimes e_{b a}.
\end{aligned}
\end{eqnarray}
It is well known that $T$ satisfies the Yang-Baxter equation
\[
T_{12} T_{13} T_{23} = T_{23} T_{13} T_{12}.
\]

We prove \eqref{eq:YBE-R} by showing that it holds when acting on the basis vectors $v_a\otimes v_b\otimes v_c$ ($a, b, c\in I$) of $V\otimes V\otimes V$. Clearly $R(v_a\otimes v_b) = T(v_a\otimes v_b)$ for all $a\ne b$.
Thus for all $a, b, c$ which are pair-wise distinct,
\[
\begin{aligned}
&R_{12} R_{13}R_{23} (v_a\otimes v_b\otimes v_c)
= T_{12} T_{13} T_{23} (v_a\otimes v_b\otimes v_c)\\
&R_{23} R_{13} R_{12}(v_a\otimes v_b\otimes v_c)
= T_{23} T_{13} T_{12}(v_a\otimes v_b\otimes v_c).
\end{aligned}
\]
Hence \eqref{eq:YBE-R} holds when acting on the vectors $v_a\otimes v_b\otimes v_c$ such that $a, b, c$ are pair-wise distinct.

Now we need to consider the actions of \eqref{eq:YBE-R} on vectors
$v_a\otimes v_b\otimes v_c$  with two or all three of $v_a,  v_b,  v_c$   being the same.  If $a=b=c$, we have
\[
\begin{aligned}
R_{12} R_{13}R_{23} (v_a\otimes v_a\otimes v_a)&=q_a^3 v_a\otimes v_a\otimes v_a \\
&= R_{23} R_{13} R_{12} (v_a\otimes v_a\otimes v_a).
\end{aligned}
\]

If $a=c\ne b$, we have the three basis vectors
$
v_a\otimes v_a\otimes v_b$, $v_a\otimes v_b\otimes v_a$, $v_b\otimes v_a\otimes v_a,
$
in each of the cases with $a<b$ or $a>b$. Consider for example the vector
$v_b\otimes v_a\otimes v_a$ with $a<b$. Then we have
\[
\begin{aligned}
R_{12} R_{13}R_{23} (v_b\otimes v_a\otimes v_a)
&= q_a v_b\otimes v_a\otimes v_a + q_a(q-q^{-1}) v_a\otimes v_b\otimes v_a\\
&\quad + q_a^2(q-q^{-1}) v_a\otimes v_a\otimes v_b, \\
R_{23} R_{13} R_{12}(v_b\otimes v_a\otimes v_a)
&= q_a v_b\otimes v_a\otimes v_a + q_a(q-q^{-1}) v_a\otimes v_b\otimes v_a\\
&\quad + \big(q_a (q-q^{-1}) +1 \big)(q-q^{-1}) v_a\otimes v_a\otimes v_b.
\end{aligned}
\]
The right hand sides of the above equations are equal since
\[
q_a^2-q_a (q-q^{-1}) -1=0, \quad \forall a\in I.
\]
We can similarly show that \eqref{eq:YBE-R} holds when acting on the other basis vectors.

This proves that $R$ satisfies \eqref{eq:YBE-R}.

\medskip
(2). To prove the second part of the lemma, we only need to show that \eqref{eq:endo} holds for the generators of $\U_q(\gl_{m, n})$.

For $x=K_b$, $b\in I$, \eqref{eq:endo} is implied by
\[
\left(\nu(K_b)\otimes\nu(K_b)\right)\Theta=\Theta\left(\nu(K_b)\otimes\nu(K_b)\right),
\]
as $R_0$ clearly commutes with $(\nu\otimes\nu)\Delta(K_b)=\nu(K_b)\otimes\nu(K_b)=(\nu\otimes\nu)\Delta'(K_b)$.
The above relation can be proved by the following computation.
\[
\begin{aligned}
&\left(\nu(K_b)\otimes\nu(K_b)\right)\Theta\left(\nu(K_b^{-1})\otimes\nu(K_b^{-1})\right)\\
&=1\otimes 1+(q-q^{-1})\sum_{c<d} \nu(K_b)e_{c d}\nu(K_b^{-1})\otimes \nu(K_b)e_{d c}\nu(K_b^{-1})\\
&=1\otimes 1+(q-q^{-1})\sum_{c<d} q_b^{\delta_{bc}}e_{c d}q_b^{-\delta_{b d}}\otimes q_b^{\delta_{b d}}e_{d c}q_b^{-\delta_{b c}}=\Theta.
\end{aligned}
\]

For $x=e_a$, we note that
\[
R_0^{-1}(\nu(e_a)\otimes 1) R_0 = \nu(e_a)\otimes \nu(k_a^{-1}), \quad
R_0^{-1}(\nu(k_a)\otimes\nu(e_a)) R_0 =1\otimes \nu(e_a).
\]
Hence \eqref{eq:endo} for $x=e_a$ is equivalent to
\[
\Theta(\nu(e_a)\otimes \nu(k_a) + 1\otimes \nu(e_a)) =
(\nu(e_a)\otimes \nu(k_a^{-1}) + 1\otimes \nu(e_a)) \Theta.
\]
Write $Q=\sum\limits_{a<b}e_{a b}\otimes e_{b a}$; then $\Theta=1\otimes 1+(q-q^{-1})Q$.  The above equation can be re-written as
\begin{eqnarray}\label{eq:Q-e}
\begin{aligned}
\nu(e_a)\otimes \nu\left(\frac{k_a - k_a^{-1}}{q-q^{-1}}\right)
&= -Q(\nu(e_a)\otimes \nu(k_a) + 1\otimes \nu(e_a)) \\
&\quad + (\nu(e_a)\otimes \nu(k_a^{-1}) + 1\otimes \nu(e_a))Q.
\end{aligned}
\end{eqnarray}
By using \eqref{eq:rep}, we can easily show that
\[
\text{LHS of \eqref{eq:Q-e}}= e_{a, a+1}\otimes (e_{a a}-e_{a+1, a+1}).
\]
To consider the right hand side,  we note that
\[
\begin{aligned}
&Q(\nu(e_a)\otimes \nu(k_a)) = \sum_{c;\,  c<a}e_{c, a+1}\otimes e_{a c}, \quad
Q(1\otimes \nu(e_a)) =\sum_{d;\,  d>a} e_{a d}\otimes e_{d, a+1}, \\
& (1\otimes \nu(e_a))Q=\sum_{c; \, c<a+1}e_{c, a+1}\otimes e_{a c}, \quad
(\nu(e_a)\otimes \nu(k_a^{-1}))Q= \sum_{d;\,  d>a+1} e_{a d}\otimes e_{d, a+1}.
\end{aligned}
\]
Using these on the right hand side of \eqref{eq:Q-e}, we obtain
\[
\text{RHS of \eqref{eq:Q-e}}= e_{a, a+1}\otimes (e_{a a}-e_{a+1, a+1}).
\]
This proves \eqref{eq:Q-e} in this case.

To prove \eqref{eq:endo} for $x=f_a$, we use
\[
R_0^{-1}(\nu(f_a)\otimes \nu(k_a^{-1})) R_0 = \nu(f_a)\otimes 1, \quad
R_0^{-1}(1\otimes \nu(f_a)) R_0 =\nu(k_a)\otimes \nu(f_a),
\]
to re-write it as
\[
\Theta (\nu(f_a)\otimes 1+\nu(k_a^{-1})\otimes \nu(f_a))= (\nu(f_a)\otimes 1+\nu(k_a)\otimes \nu(f_a))\Theta.
\]
This is equivalent to
\begin{eqnarray}\label{eq:Q-f}
\begin{aligned}
\nu\left(\frac{k_a-k_a^{-1}}{q-q^{-1}}\right)\otimes\nu(f_a)
&=Q (\nu(f_a)\otimes 1+\nu(k_a^{-1})\otimes \nu(f_a))\\
&\quad - (\nu(f_a)\otimes 1+\nu(k_a)\otimes \nu(f_a))Q.
\end{aligned}
\end{eqnarray}
Similar calculations like those in the case of $e_a$ can show that both side of the above express are equal to $(e_{a a} - e_{a+1, a+1})\otimes e_{a+1, a}$.
This completes the proof.
\end{proof}

Let $P: V\otimes V\longrightarrow V\otimes V$,  $v\otimes v'\mapsto v'\otimes v$ for all $v, v' \in V$, be the permutation map, which can be expressed in terms of the matrix units as $P=\sum\limits_{a, b} e_{a b}\otimes e_{b a}$.  Define
\begin{eqnarray}\label{eq:R-check}
\check{R}:=P R.
\end{eqnarray}
The following result immediately follows from Lemma \ref{lem:R-matrix} and Example \ref{ex:VV}.
\begin{corollary}\label{cor:R-check}
The matrix $\check{R}$ is an invertible element of $\End_{\U_q(\gl_{m, n})}(V\otimes V)$. It satisfies the Yang-Baxter equation
\begin{eqnarray}\label{eq:YBE}
(\check{R}\otimes 1)(1\otimes \check{R})(\check{R}\otimes 1)=(1\otimes \check{R})(\check{R}\otimes 1)(1\otimes \check{R}),
\end{eqnarray}
and the quadratic relation
\begin{eqnarray}\label{eq:Hecke}
(\check{R}-q)(\check{R}+q^{-1})=0.
\end{eqnarray}
\end{corollary}
\begin{proof}
All the statements are clear from Lemma \ref{lem:R-matrix} except the spectral decomposition
\eqref{eq:Hecke}.

By Example \ref{ex:VV},  there are two simple submodules $L_s$ and $L_a$ in $V\otimes V$.
They are  eigenspaces of $\check{R}\in\End_{\U_q(\gl_{m, n})}(V\otimes V)$.
We can determine the eigenvalues by considering the action of $\check{R}$ on the respective highest weight vectors.

Recall that the highest weight vector of $L_s$ is $v_1\otimes v_1$.
It is immediate to calculate $\check{R}(v_1\otimes v_1) = q v_1\otimes v_1$.

The highest weight vector of $L_a$ is $w= v_1\otimes v_2 -q^{-1} v_2\otimes v_1$.  We have $\Theta w = q^{-2} w'$, where $w'=v_1\otimes v_2 -q v_2\otimes v_1$.  Now $R_0 w'=w'$ and $Pw' = - q w$. Hence $\check{R}(w)= -q^{-1} w$. This proves \eqref{eq:Hecke}.
\end{proof}

Finally we return to the problem raised in the last section about isomorphisms of tensor product modules defined relative to $\Delta$ and $\Delta'$.   Recall the definition of $\Delta^{(r-1)}$  given in \eqref{eq:Delta-r}. We can similarly define ${\Delta'}^{(r-1)}$.
The following result is an easy corollary of Lemma \ref{lem:R-matrix}.
\begin{corollary}\label{cor:iso}
Denote the $\U_q(\gl_{m, n})$-module $V^{\otimes r}$ defined relative to $\Delta^{(r-1)}$  (resp. ${\Delta'}^{(r-1)}$) by $(V^{\otimes r}, \Delta^{(r-1)})$ (resp. $(V^{\otimes r}, {\Delta'}^{(r-1)})$).  Then
$(V^{\otimes r}, \Delta^{(r-1)})$  is isomorphic to  $(V^{\otimes r}, {\Delta'}^{(r-1)})$ for any $r\ge 2$.
\end{corollary}
\begin{proof}
For $r=2$, it is obvious from part (2) of  Lemma \ref{lem:R-matrix} that the isomorphism is provided by the $R$-matrix.  For $r>2$,  an isomorphism is given by
$R_{1 r} R_{2 r} \dots R_{r-1, r}$.
\end{proof}

\begin{remark}
An interesting problem is the decomposition of $V^{\otimes r}$ for all $r$.
A first step in studying this problem is to understand the endomorphism algebras
$\End_{\U_q(\gl_{m, n})}(V^{\otimes r})$; see Remark \ref{rem:double-comm} for further discussions.
\end{remark}

\subsection{The case of $\U_q(\fsl_{2, 1})$ -- an example}\label{sect:reps-sl21}
In this section, we study the representation theory of $\U_q(\fsl_{2, 1})$ in more depth.
An explicit basis will be constructed for each finite dimensional simple $\U_q(\fsl_{2, 1})$-module.

We will need the following result later.
\begin{lemma}\label{lem:ff}
Let $F:=f_1 f_2 - q f_2 f_1$. Then the following relations  hold
\begin{eqnarray}
&f_1 F  = q^{-1} F f_1, \quad   f_2 F = - q^{-1}Ff_2,  \quad F^2=0; \label{eq:ff-1}\\
&e_1 F - F e_1 = f_2 k_1^{-1}, \quad e_2 F - F e_2 =-q f_1 k_2, \label{eq:ff-2}\\
&f_1^k f_2 = [k]_q F f_1^{k-1} + q^k f_2 f_1^k.  \label{eq:ff-3}
\end{eqnarray}
\end{lemma}
\begin{proof}
The proof is  straightforward; we omit the details.
\end{proof}

Now given any pair $\lambda:=(\lambda_1, \lambda_2)$ of scalars $\lambda_i\in\C(q)$ which are both nonzero, let  $L(\lambda)$ be the simple $\U_q(\fsl_{2, 1})$-module with the highest weight vector  $v_\lambda$ such that
\[
e_i v_\lambda=0, \quad  k_iv_\lambda=\lambda_i v_\lambda, \quad i=1, 2.
\]
It follows from Theorem \ref{thm:fd-irreps} that
\begin{lemma}
The simple $\U_q(\fsl_{2,1})$ module $L(\lambda)$ with highest weight $\lambda=(\lambda_1, \lambda_2)$ is finite dimensional if and only if $\lambda_1 = \pm q^\ell$ for some nonnegative integer $\ell$.
\end{lemma}

\begin{lemma} \label{lem:typical}
Let $\lambda=(\lambda_1, \lambda_2)$ with $\lambda_1=\pm q^\ell$ for some nonnegative integer $\ell$. Then the simple $\U_q(\fsl_{2, 1})$-module $L(\lambda)$ has dimension $4(\ell +1)$ if and only if
\begin{eqnarray}\label{eq:typical}
(q \lambda_1\lambda_2 - q^{-1} \lambda_1^{-1}\lambda_2^{-1})(\lambda_2 - \lambda_2^{-1}) \ne 0.
\end{eqnarray}
In this case,  the following set of vectors forms a basis of $L(\lambda)$.
\begin{eqnarray}\label{eq:span-set}
f_1^k v_\lambda,  \ f_2 f_1^k v_\lambda, \ F f_1^k v_\lambda, \ F f_2 f_1^k v_\lambda, \quad k=0, 1, \dots, \ell.
\end{eqnarray}
\end{lemma}
\begin{proof}
Note that  $\dim L(\lambda)<4 (\ell+1)$ if and only if the vectors given in \eqref{eq:span-set} are
linearly dependent. In this case, taking any nontrivial linear combination of these vectors which vanishes, we may apply $f_2$, $F$ or both to it to obtain
\[
F f_2 f_1^{k_0} v_\lambda =0
\]
for  some nonnegative integer ${k_0}\le \ell$, by using the facts that $f_2^2=0$ (see \eqref{eq:relat-8}),  $f_2 F = -q^{-1} F f_2$ and $F^2=0$ (see \eqref{eq:ff-1}).

By using the first relation in \eqref{eq:ff-2} and also  $f_2^2=0$ again, we obtain
$e_1 F f_2 f_1^{k_0} v_\lambda  =F f_2 e_1 f_1^{k_0} v_\lambda$, and hence
$e_1^{k_0} F f_2 f_1^{k_0} v_\lambda  =F f_2 e_1^{k_0} f_1^{k_0} v_\lambda$.
By repeatedly applying the well known formula
\begin{eqnarray}\label{eq:e-fpower}
[e_1, f_1^k] = [k]_q f_1^{k-1} \frac{k_1 q^{1-k} - k_1^{-1}q^{k-1}}{q-q^{-1}} \quad \text{with } \ [k]_q:=\frac{q^k - q^{-k}}{q-q^{-1}},
\end{eqnarray}
we obtain
\[
e_1^{k_0} f_1^{k_0} v_\lambda = c_{k_0} v_\lambda, \quad   c_{k_0}:=\prod_{k=1}^{k_0}[k]_q[\ell+1-k]_q\ne 0.
\]
Hence $e_1^{k_0} F f_2 f_1^{k_0} v_\lambda = c_{k_0} F f_2 v_\lambda=0$,  i.e.,
\begin{eqnarray}\label{eq:bottom}
F f_2 v_\lambda=0.
\end{eqnarray}
This is the necessary and sufficient condition for $\dim L(\lambda)<4 (\ell+1)$.

The vector $F f_2 v_\lambda$ is clearly not the highest weight vector. Thus it vanishes if and only if
\[
e_2 F f_2 v_\lambda= e_1 F f_2 v_\lambda=0.
\]
The second condition is trivial in view of the first relation in \eqref{eq:ff-2} and the fact that $f_2^2=0$.  From the first one, we obtain
\[
v:=\left(\frac{\lambda_2 - \lambda_2^{-1}}{q-q^{-1}} F  + q \lambda_2 f_1 f_2\right)v_\lambda=0.
\]
Again this holds if and only if $e_1 v=e_2 v=0$. The second condition is always true. The first leads to
\[
(q \lambda_1\lambda_2 - q^{-1} \lambda_1^{-1}\lambda_2^{-1}) f_2 v_\lambda=0.
\]
This holds if and only if
\[
(q \lambda_1\lambda_2 - q^{-1} \lambda_1^{-1}\lambda_2^{-1})e_2 f_2 v_\lambda = (q \lambda_1\lambda_2 - q^{-1} \lambda_1^{-1}\lambda_2^{-1})(\lambda_2 - \lambda_2^{-1}) v_\lambda =0.
\]
Thus we conclude that in order  for $\dim L(\lambda)< 4(\ell+1)$, the necessary and sufficient condition is
\begin{eqnarray}\label{eq:atypical}
(q \lambda_1\lambda_2 - q^{-1} \lambda_1^{-1}\lambda_2^{-1})(\lambda_2 - \lambda_2^{-1}) =0.
\end{eqnarray}
This completes the proof  of the first statement of the lemma.   The second statement  easily follows.
\end{proof}
Call \eqref{eq:atypical} the atypicality condition by adopting the terminology from the representation theory of Lie superalgebras.

\begin{lemma}\label{lem:atypical}
Continue to assume that $\lambda_1=\pm q^{\ell}$ for some nonnegative integer $\ell$, and
 assume that the atypicality condition \eqref{eq:atypical} holds. Then $\lambda$ belongs to one of the following mutually exclusive cases: (a)  $\lambda_2=\pm 1$, or
(b) $\lambda_2 = \pm q^{-1}\lambda_1^{-1}$.

 If $\lambda$ belongs to case (a), $L(\lambda)$ is $2\ell +1$  dimensional with a basis
\begin{eqnarray}\label{eq:basis-1}
\{f_1^{j}v_\lambda,  F f_1^k v_\lambda \mid 0\le j\le \ell,  \  0\le k \le \ell-1\}.
\end{eqnarray}

If $\lambda$ belongs to case (b), $L(\lambda)$ is $2(\ell +1)+1$  dimensional with a basis
\begin{eqnarray}\label{eq:basis-2}
\{f_1^k v_\lambda,  f_2 f_1^k v_\lambda, F f_1^\ell v_\lambda \mid 0\le k\le \ell\}.
\end{eqnarray}
\end{lemma}
\begin{proof}
The first statement is clear.

To describe  the structure of $L(\lambda)$,  we recall from the proof of Lemma \ref{lem:typical} that $F f_2 f_1^k v_\lambda=0$ for all $k\ge 0$.  What we need to sort out is the linear dependence of the vectors $F f_1^k v_\lambda$ and $f_2 f_1^k v_\lambda$.  We consider the two cases separately.

In case (a), we have $f_2 v_\lambda=0$. This leads to
$
f_1 f_2 v_\lambda = F v_\lambda + q f_2 f_1 v_\lambda=0,
$
that is,
\[
F v_\lambda = - q f_2 f_1 v_\lambda.
\]
More generally, we have the following relation for all $k\ge 0$,
\begin{eqnarray}\label{eq:ffv-1}
Ff_1^k v_\lambda= - \frac{q^{k+1}}{[k+1]_q}f_2 f_1^{k+1} v_\lambda, \quad k\ge 0,
\end{eqnarray}
which we prove by induction on $k$.  Note that the first relation in \eqref{eq:ff-1} and the definition of $F$ respectively lead to
$f_1Ff_1^{k-1} = q^{-1}F f_1$ and $f_1 f_2 f_1^k = F  f_1^k + q f_2 f_1^{k+1}$.  Assuming
\eqref{eq:ffv-1} holds for $k-1$, we have
\[
q^{-1} F f_1^k v_\lambda= - \frac{q^k}{[k]_q} ( F + q f_2 f_1) f_1^k v_\lambda.
\]
Moving the term $- \frac{q^k}{[k]_q} F f_1^k v_\lambda$ on the right hand side to the left and using $q^{-1} [k]_1 + q^k= [k+1]_q$,  we obtain \eqref{eq:ffv-1}.

The relation \eqref{eq:ffv-1} shows that the vectors in \eqref{eq:basis-1} span $L(\lambda)$.

If we also have $\ell=0$, then $f_1v_\lambda=0$ and hence $Fv_\lambda=0$.  If $\ell>0$, we have
$e_2 F v_\lambda= - q \frac{k_2-k_2^{-1}}{q-q^{-1}}f_1v_\lambda=  - q f_1v_\lambda\ne0$, and
$e_1 Fv_\lambda = -q f_2 e_1 f_1 v_\lambda=-q[\ell]_q f_2 v_\lambda=0$.  Furthermore, $k_1 F v_\lambda = q^{\ell-1} F v_\lambda$. Hence $F v_\lambda$ is a highest weight vector of the $\U_q(\fsl_2)$ subalgebra with highest weight $\pm q^{\ell-1}$, which generates the $\ell$-dimensional simple $\U_q(\fsl_2)$-submodule with  basis vectors $F f_1^k v_\lambda$ for $k=0, 1, \dots, \ell-1$.

This in particular shows that the set  $B^1:=\{F f_1^k v_\lambda\mid 0\le k\le \ell-1\}$ is linearly independent.  It is  clear that $B^0:=\{f_1^k v_\lambda\mid 0\le k\le \ell\}$ is also linearly independent. The $k_2$ eigenvalues of the vectors in $B^1$ are the negatives of those in
$B^0\backslash\{v_\lambda\}$, hence the two sets are linearly independent.
This proves that \eqref{eq:basis-1} is a basis of $L(\lambda)$.

Now we consider case (b).

If $\ell=0$, it is easy to see that $L(\lambda)$ has a basis $\{ v_\lambda,  f_2v_\lambda,  F v_\lambda=f_1 f_2 v_\lambda\}$.

Assume $\ell>0$. Then $e_2 F f_2 v_\lambda=0$ leads to
$
[\ell]_q F v_\lambda = q^{-\ell} f_2 f_1 v_\lambda.
$
A similar inductive proof as that for \eqref{eq:ffv-1} shows that
\begin{eqnarray}\label{eq:ffv-2}
[\ell-k]_q F f_1^k v_\lambda = q^{k-\ell} f_2 f_1^{k+1} v_\lambda, \quad k=0, 1, \dots, \ell-1,
\end{eqnarray}
where the vector on the right hand side is clearly nonzero.
In particular, for $k=\ell-1$, we have $Ff_1^{\ell-1} v_\lambda = q^{-1} f_2 f_1^\ell v_\lambda$.
This leads to
$
Ff_1^\ell v_\lambda = f_1 f_2 f_1^\ell v_\lambda.
$
The right hand side is nonzero, as
\[
\begin{aligned}
e_1 f_1 f_2 f_1^\ell v_\lambda
	&= \pm([\ell-1]_q f_2 f_1^\ell + [\ell]_q f_1 f_2 f_1^{\ell-1})v_\lambda \\
	&=\pm([\ell-1]_q f_2 f_1^\ell + [\ell]_q F f_1^{\ell-1} +[\ell]_q  q f_2 f_1^\ell)v_\lambda \\
	&= \pm([\ell-1]_q  +  (q+q^{-1}) [\ell]_q)f_2 f_1^\ell v_\lambda\ne 0.
\end{aligned}
\]

We have now proved that the vectors in \eqref{eq:basis-2} are all nonzero and span $L(\lambda)$.
They are linearly independent by similar weight considerations as in case (a).

This completes the proof of the lemma.
\end{proof}

\section{Invariants of the degenerate quantum general linear group}
\subsection{Some invariant theory}\label{sect:inv}
Recall that the antipode $S$ of $\U_q(\gl_{m, n})$ is an algebraic anti-automorphism, thus its square is an automorphism. It satisfies
\[
S^2(K_b)= K_b, \quad S^2(e_a)= k_a e_a k_a^{-1}, \quad  S^2(f_a)= k_a f_a k_a^{-1}, \quad \forall a\in I', \ b\in I.
\]
\begin{lemma} The square of the antipode  of $\U_q(\gl_{m, n})$ is an inner automorphism, namely, there exists an invertible element $K_{2\rho}\in \U_q(\gl_{m, n})$ such that
\begin{eqnarray}\label{eq:S-sq}
S^2(x) = K_{2\rho} x K_{2\rho}^{-1}, \quad \forall x\in \U_q(\gl_{m,n}).
\end{eqnarray}
Such a $K_{2\rho}$ can be constructed as follows:
\begin{eqnarray}
K_{2\rho} = \left\{
\begin{array}{l l}
K'_{2\rho}, & \text{if $m+n$ is even}, \\
K'_{2\rho} K', & \text{otherwise},
\end{array}
\right.
\end{eqnarray}
with
$
K'_{2\rho} = \prod_{a=1}^mK_a^{m-n+1-2a}\prod_{\mu=1}^n K_{m+\mu}^{m+n+1-2\mu}
$
and
$
K' =  \prod_{a=1}^m K_a\prod_{\mu=1}^n K_{m+\mu}^{-1}.
$
\end{lemma}
\begin{proof} Since $S^2$ is an algebra automorphism,  we only need to prove that with the $K_{2\rho}$ constructed,  equation \eqref{eq:S-sq} for all the generators of $\U_q(\gl_{m,n})$.
It is obvious that \eqref{eq:S-sq} holds for all $K_b$.  We can also easily prove
that \eqref{eq:S-sq} holds for all $e_a$ and $f_a$ with $a\ne m$ by noting that
\[
\begin{aligned}
&K'_{2\rho} e_a {K'}_{2\rho}^{-1} = k_a e_a k_a^{-1},  \quad K'_{2\rho} f_a {K'}_{2\rho}^{-1} = k_a f_a k_a^{-1}, \\
&K' e_a {K'}^{-1} = e_a,  \quad K' f_a {K'}^{-1} = f_a, \quad a\ne m.
\end{aligned}
\]
To consider $e_m$ and $f_m$, we note that
\[
\begin{aligned}
&K'_{2\rho} e_m {K'}_{2\rho}^{-1} = k_m^{1-m-n} e_m k_m^{-1+m+n}= (-1)^{m+n+1} e_m,  \\
&K'_{2\rho} f_m {K'}_{2\rho}^{-1} = k_m^{1-m-n} f_m k_m^{-1+m+n}= (-1)^{m+n+1} f_m,  \\
&K' e_m {K'}^{-1} = k_m e_m k_m^{-1} = - e_m,  \\
& K' f_m {K'}^{-1} =k_m f_m k_m^{-1}=- f_m.
\end{aligned}
\]
Hence $K_{2\rho} e_m K_{2\rho}^{-1} = - e_m = k_m e_m k_m^{-1}$,  and $
 K_{2\rho} f_m K_{2\rho}^{-1} = - f_m=k_m f_m k_m^{-1}.$
This completes the proof of the lemma.
\end{proof}

Let $M$ be a finite dimensional $\U_q(\gl_{m, n})$-module,  and denote the associated representation by $\pi: \U_q(\gl_{m, n}) \longrightarrow \End_{\C(q)}(M)$.  We have the following quantum adjoint action of $\U_q(\gl_{m, n})$ on $\End_{\C(q)}(M)$ 
\[
\begin{aligned}
&\U_q(\gl_{m, n}) \otimes \End_{\C(q)}(M) \longrightarrow \End_{\C(q)}(M), \\
&x\otimes A\mapsto ad_x(A):=\sum_{(x)} \pi(x_{(1)}) A \pi(S(x_{(2)})),
\end{aligned}
\]
where we have used Sweedler's notation $\Delta(x)=\sum_{(x)} x_{(1)}\otimes S(x_{(2)})$ for the co-multiplication.  We denote
 the endomorphism algebra of $M$ over $\U_q(\fgl_{m, n})$ by
\[
\End_{\U_q(\fgl_{m, n})}(M):=\{A\in  \End_{\C(q)}(M)\mid \pi(x) A - A\pi(x) =0, \ \forall x\in \U_q(\fgl_{m, n})\}.
\]

Define the quantum trace on $\End_{\C(q)}(M)$ by
\begin{eqnarray}\label{eq:q-trace}
\tau_{M}: \End_{\C(q)}(M)\longrightarrow\C(q), \quad A\mapsto tr(\pi(K_{2\rho})A),
\end{eqnarray}
where $tr$ is the trace over $M$.

The following lemma follows from simple facts in the theory of Hopf algebras. 
\begin{lemma}\label{lem:q-trace}
Keep notation above. 
\begin{enumerate}
\item An element $A$ of $\End_{\C(q)}(M)$ belongs to $\End_{\U_q(\fgl_{m, n})}(M)$ if and ony if 
\[ 
ad_x(A) =\epsilon(x) A, \quad \forall x\in \U_q(\fgl_{m, n}).
\]
\item  The quantum trace $\tau_M$ is $ad$-invariant in the sense that
\[
\tau_M(ad_x(A)) = \epsilon(x) \tau_M(A), \quad \forall x\in\U_q(\gl_{m, n}), \ A\in \End_{\C(q)}(M).
\]
\end{enumerate}
\end{lemma}
\begin{proof}  The proof of the lemma is simple but worth knowning.  

Consider the first statement. 
 If $A\in\End_{\U_q(\fgl_{m, n})}(M)$,  we have
\[
\begin{aligned}
ad_x(A)=\sum_{(x)} \pi(x_{(1)}) A \pi(S(x_{(2)}))&= \sum_{(x)} \pi(x_{(1)})\pi(S(x_{(2)})) A = \epsilon(x) A,
\end{aligned}
\]
for all $x\in\U_q(\gl_{m,n})$.  
To prove the opposite direction, we note that for any $A\in \End_{\C(q)}(M)$, 
 the defining properties of the antipode leads to   
\[
\pi(x) A=\sum_{(x)} ad_{x_{(1)}}(A) \pi(x_{(2)}), \quad \forall x\in \U_q(\fgl_{m, n}).
\]
If $ad_x(A)=\epsilon(x)A$ for $x\in \U_q(\fgl_{m, n})$, then 
\[
\sum_{(x)} ad_{x_{(1)}}(A) \pi(x_{(2)})=\sum_{(x)} \epsilon(x_{(1)})A \pi(x_{(2)}) =A \pi(x). 
\]
Hence $ \pi(x) A- A \pi(x)=0$, and $A\in\End_{\U_q(\fgl_{m, n})}(M)$. 
This proves the first statement. 

To prove the second statement, we note that 
\[
\begin{aligned}
\tau_M(ad_x(A)) &= \sum_{(x)}tr(\pi(K_{2\rho})\pi(x_{(1)}) A \pi(S(x_{(2)})))\\
&= \sum_{(x)}tr( \pi(S^2(x_{(1)})) \pi(K_{2\rho})A \pi(S(x_{(2)}))). \qquad \text{(by \eqref{eq:S-sq})}
\end{aligned}
\]
Using the cyclic property of the trace,  we can rewrite the right hand side as
\[\sum_{(x)}tr(\pi(K_{2\rho})A \pi(S(x_{(2)}) \pi(S^2(x_{(1)})) ).\]
Since $S$ is an anti-automorphism, this can be rewritten as
\[
 \sum_{(x)}tr(\pi(K_{2\rho})A \pi(S(S(x_{(1)})x_{(2)}))),
\]
which is equal to
$
\epsilon(x) tr(\pi(K_{2\rho})A)=\epsilon(x) \tau_M(A)
$
by the defining property \eqref{eq:S-def} of $S$.  
This completes the proof.
\end{proof}

Denote by $\id_M$ the identity map on the $\U_q(\gl_{m, n})$-module $M$. Let
\[
\dim_q(M)=\tau_M(\id_M),
\]
and call it the quantum dimension of $M$.  It is always well defined for finite dimensional $\U_q(\gl_{m, n})$-modules.

\begin{example} The quantum dimension of $V=\C(q)^{m+n}$ is given by
\[
\dim_q(V)=\left\{
\begin{array}{l l}
[m-n]_q, &\text{if $m+n$ is even}, \\
q [m-n]_q, &\text{otherwise}.
\end{array}
\right.
\]
\end{example}

Let $V_1$ and $V_2$ be $\U_q(\gl_{m, n})$-modules, and denote by
$
\pi_i: \U_q(\gl_{m, n}) \longrightarrow \End_{\C(q)}(V_i)
$
($i=1, 2$)
the corresponding representations respectively.   Then $V_1\otimes V_2$ forms a
$\U_q(\gl_{m, n})$-module with the associated representation
$
(\pi_1\otimes\pi_2)\Delta: \U_q(\gl_{m, n})\longrightarrow \End_{\C(q)}(V_1\otimes V_2).
$
	
The following result can be deduced from \cite[Proposition1]{ZGB1}.
\begin{theorem}\label{thm:inv}
Keep notation above. Define the linear map
\[
 \Phi:  \End_{\C(q)}(V_1\otimes V_2)\longrightarrow \End_{\C(q)}(V_1), \quad
\Gamma\mapsto(\id\otimes \tau_{V_2})(\Gamma),
\]
where $\id$  is the identity map on $\End_{\C(q)}(V_1)$.
If $\Gamma\in\End_{\U_q(\gl_{m, n})}(V_1\otimes V_2)$,  then $\Phi(\Gamma)\in\End_{\U_q(\gl_{m, n})}(V_1)$.
\end{theorem}	
\begin{proof} Note that $\Gamma\in\End_{\C(q)}(V_1\otimes V_2)$ belongs to $\End_{\U_q(\gl_{m, n})}(V_1\otimes V_2)$ if and only if
\begin{eqnarray}\label{eq:ad-2}
\sum_{(x)} (\pi_1\otimes\pi_2)\Delta(x_{(1)})\Gamma (\pi_1\otimes\pi_2)\Delta(S(x_{(2)}))=\epsilon(x) \Gamma, \quad \forall x\in\U_q(\gl_{m, n}).
\end{eqnarray}
The left hand side can be rewritten as
\[
\sum_{(x)} \pi_1(x_{(1)})\otimes\pi_2(x_{(2)})\Gamma \pi_1(S(x_{(4)}))\otimes
\pi_2(S(x_{(3)})).
\]
Applying $\Phi$ to it and using Lemma \ref{lem:q-trace}, we obtain
\[
\begin{aligned}
\sum_{(x)} \pi_1(x_{(1)})\epsilon(x_{(2)})\Phi(\Gamma) \pi_1(S(x_{(3)}))
=\sum_{(x)} \pi_1(x_{(1)})\Phi(\Gamma) \pi_1(S(x_{(2)})) = ad_x(\Phi(\Gamma)).
\end{aligned}
\]
Hence $\Phi$ maps \eqref{eq:ad-2} to
\[
ad_x(\Phi(\Gamma)) = \epsilon(x)\Phi(\Gamma), \quad \forall x\in\U_q(\gl_{m, n}).
\]
This implies that $\Phi(\Gamma)\in\End_{\U_q(\gl_{m, n})}(V_1)$.
\end{proof}

\begin{remark}
Lemma \ref{eq:q-trace} and Theorem \ref{thm:inv} are true for any Hopf algebra with the square of the antipode being an inner automorphism.
\end{remark}

\subsection{An application -- the HOMFLY polynomial}\label{sect:app}

It is well-known that any $\check{R}$-matrix satisfying the Yang-Baxter equation leads to a
representation of the braid group, and under favourable conditions, a topogical invariant of link invariant  can be constructed from the representation. We now construct a link invariant from the $\check{R}$-matrix given by \eqref{eq:R}.

Recall that the braid group $B_r$ on $r$ strings is generated by
$b_1, b_2, \dots, b_{r-1}$ subject to the relations
\[
\begin{aligned}
&b_i b_j = b_j b_i, \quad \text{if $|i-j|>1$}, \\
&b_i b_{i+1} b_i =   b_{i+1} b_i  b_{i+1}, \quad \text{for $i<r-1$}.
\end{aligned}
\]
It has the chain of subgroups $B_2<B_3<\dots < B_{r-1}< B_r$, where
$B_{k-1}$ is the subgroup of $B_k$ generated by $b_i$ with $1\le i\le k-2$.

A closely related algebra which will be relevant here is the Hecke algebra $H_r(q)$ of type A, which is generated by $T_i$ with $i=1, 2, \dots, r-1$ subject to the relations
\[
\begin{aligned}
&T_i T_j = T_j T_i, \quad \text{if $|i-j|>1$}, \\
&T_i T_{i+1} T_i =   T_{i+1} T_i  T_{i+1},  \\
&(T_i-q)(T_i+q^{-1})=0.
\end{aligned}
\]
It is isomorphic to the quotient of the group algebra $\C(q)B_r$ of $B_r$ by the two-sided ideal generated by $(b_i-q)(b_i+q^{-1})$ for all $i$. We denote by $\psi_r: \C(q)B_r\longrightarrow H_r(q)$ the canonical surjection.

Write $E_r : =\End_{\U_q(\gl_{m, n})}(V^{\otimes r})$.  Recall that the $\check{R}$-matrix given by \eqref{eq:R-check}  belongs to $E_2$.
The proof of the following result is routine.
\begin{proposition} The following map defines a representation of the group algebra of the braid group $B_r$,
\[
\nu_r:  \C(q)B_r \longrightarrow  E_r, \quad \nu_r(b_i)= \id_V^{\otimes (i-1)} \otimes \check{R}\otimes \id_V^{\otimes (r-i-1)}.
\]
It factors through the Hecke algebra $H_r(q)$, that is, we have the following commutative diagram.
\[
\xymatrix{
\C(q)B_r\ar[d]_{\psi_r}\ar[r]^{\nu_r} &E_r&\\
H_r(q) \ar[ru]&
}
\]
\end{proposition}
\begin{proof} Since $\check{R}\in E_2$, we have $\nu_r(b_i)\in E_r$ for all $i$.
The $\check{R}$-matrix satisfies the Yang-Baxter equation \eqref{eq:YBE}, hence it immediately follows that $\nu_r$ defines a representation of the braid group.   
By equation \eqref{eq:Hecke},
\begin{eqnarray}\label{eq:Hecke-T}
(\nu_r(b_i)-q)(\nu_r(b_i)+q^{-1})=0, \quad \forall i.
\end{eqnarray}
Thus this braid group representation factors through the Hecke algebra $H_r(q)$.
\end{proof}

\begin{remark}\label{rem:double-comm}
A natural question is whether $E_r=\nu_r(\C(q)B_r)$.  The analogous question has an affirmative answer for the ordinary quantum general linear group (see, e.g., \cite{J2, LZ}) and quantum general linear supergroup \cite{Z98, Zy}, and we expect the same answer in the present case.
\end{remark}

Hereafter  we assume that $m\ne n$, thus $\dim_q(V)\ne 0$.  For each $r$, we define a map
\begin{eqnarray}\label{eq:Markov}
\phi_r: B_r \longrightarrow \C(q),  \quad b\mapsto \frac{\tau_V^{\otimes r}(\nu_r(b))}{\dim_q(V)^r}.
\end{eqnarray}
For any $b$ in the subgroup $B_{r-1}$ of $B_r$ generated by $b_i$ for $1\le i<r-1$, we have
\[
\phi_r(b) = \phi_{r-1}(b).
\]

In analogy with \cite[Proposition 3]{ZGB1} (see also \cite{LGZ, RT}), we have the following result.
\begin{theorem} \label{thm:knot-inv} The maps $\phi_r$ have the following Markov properties
\[
\begin{aligned}
I. \quad & \phi_r(b b') = \phi_r(b'b), && \forall b, b'\in B_r,\\
II. \quad & \phi_r(b b_{r-1}) = \frac{q^{m-n}}{[m-n]_q}\phi_r(b), &&\\
& \phi_r(b b_{r-1}^{- 1}) = \frac{q^{n-m}}{[m-n]_q}\phi_r(b),
	&\quad&b\in B_{r-1}<B_r.
\end{aligned}
\]
Thus they give rise to topological invariant of framed links, which is the HOMFLY polynomial.
\end{theorem}
\begin{proof}
If we can prove that maps $\phi_r$ have the Markov properties, then they give rise to a link invariant. Now \eqref{eq:Hecke-T} leads to a skein relation which is the same as that defining the HOMFLY polynomial of framed links \cite{HOMFLY} with $m-n$ as the additional parameter.

The proof of the Markov properties $\phi_r$ is rather standard \cite[Proposition 3]{ZGB1} (except the computation of the scalar factors in part (2)),  thus we will only given an outline of the proof.

Since $\check{R}\in \End_{\U_q(\gl_{m, n})}(V\otimes V)$, we have $\nu_r(b)\in \End_{\U_q(\gl_{m, n})}(V^{\otimes r})$ for any $b\in B_r$.  This in particular implies that $\nu_r(b)$ commutes with $\nu(K_{2\rho})^{\otimes r}$. Hence the cyclic property of $\phi_r$ (i.e., property I) follows.

It follows from Theorem \eqref{thm:inv} that $(\id_V\otimes\tau_V)(\check{R}^{\pm 1})=\gamma_\pm\id_V$ for some scalars $\gamma_\pm$.
Hence  property II follows but for the scalars $ \frac{\gamma_\pm}{\dim_q(V)}$. Now we need to show that
\[
 \frac{\gamma_\pm}{\dim_q(V)}= \frac{q^{\pm (m-n)}}{[m-n]_q}.
\]

Note that  $\gamma_\pm$ can be computed as follows. For each $d\in I$, introduce the projection operator $p_d: V\longrightarrow \C(q)v_d$, and consider
\[
\left(p_c\otimes p_d\nu(K_{2\rho})\right)\check{R}^{\pm 1}(v_c\otimes v_d) = \beta_{c d}^\pm v_c\otimes v_d,
\]
where $\beta_{c d}^\pm$ are scalars. Then
\[
\gamma_\pm = \sum_{d\in I} \beta_{c d}^\pm, \quad \text{which are independent of $c$}.
\]
Direct calculations using the explicit formula for $\check{R}$ yield
\[
\begin{aligned}
\beta_{1 d}^+ &=
\left\{\begin{array}{l l}
q^{m-n}\delta_{d 1}, & \text{if $m+n$ is even}, \\
q^{m-n+1}\delta_{d 1},& \text{if $m+n$ is odd}.
\end{array}
\right.
\end{aligned}
\]
Hence in both cases, we have
$
\frac{\gamma_+}{\dim_q(V)} = \frac{ q^{m-n}}{[m-n]_q}.
$

Similar computation leads to
\[
\begin{aligned}
\beta_{m+n,  d}^- &=
\left\{\begin{array}{l l}
q^{-m+n}\delta_{d, m+n}, & \text{if $m+n$ is even}, \\
q^{-m+n+1}\delta_{d, m+n},& \text{if $m+n$ is odd}.
\end{array}
\right.\\
\end{aligned}
\]
Hence
$
\frac{\gamma_-}{\dim_q(V)} = \frac{ q^{-m+n}}{[m-n]_q}.
$
This completes the proof of the theorem.
\end{proof}

\begin{remark} It is possible to generalise the above construction of link invariant to the case with $m=n$ by using results of \cite{LGZ}.
\end{remark}

\section{Comments}
In this final section of the paper, we comment upon a possible generalisation of the results to other classical Lie algebras, and point out some similarities between degenerate quantum groups and quantum supergroups.

\subsection{Degenerate quantum groups of other types}\label{sect:general}

Recall that the definition of a Drinfeld-Jimbo quantum group can be simply encoded in the Dynkin diagram of the corresponding Lie algebra.  We can also mimic this for the degenerate quantum group $\U_q(\fsl_{m, n})$. This will then suggest a possible generalisation of $\U_q(\fsl_{m, n})$ to degenerate quantum groups of other classical types.

Draw $\ell$ nodes ordered from left to right, and colour all nodes white except for the $m$-th one, which is grey. If we connect the neighbouring nodes by one line, we obtain a generalised Dynkin diagram of $A$ type in Figure \ref{fig:A}.  We can similarly draw generalised Dynkin diagrams of $B$, $C$ and $D$ types, as shown in Figure \ref{fig:B}, Figure \ref{fig:C} and Figure \ref{fig:D} respectively.

\begin{figure}[h]
\begin{picture}(200, 20)(40,0)
\put(47, 10){\tiny $1$}
 \put(110, 10){\tiny $m-1$}
\put(145, 10){\tiny $m$}
\put(170, 10){\tiny $m+1$}
\put(217, 10){\tiny $\ell$}
\put(50, 0){\circle{10}}
\put(55, 0){\line(1, 0){10}}
\put(65, -1){...}
\put(75, 0){\line(1, 0){10}}
\put(90, 0){\circle{10}}
\put(95, 0){\line(1, 0){20}}
 \put(120, 0){\circle{10}}
\put(125, 0){\line(1, 0){20}}
{\color{gray} \put(150, 0){\circle*{10}}}
\put(155, 0){\line(1, 0){20}}
\put(180, 0){\circle{10}}
\put(185, 0){\line(1, 0){10}}
\put(195, -1){...}
\put(205, 0){\line(1, 0){10}}
\put(220, 0){\circle{10}}
\end{picture}
\caption{Degenerate quantum group of type $A$}
\label{fig:A}

\begin{picture}(200, 30)(10,0)
\put(17, 10){\tiny $1$}
\put(80, 10){\tiny $m-1$}
 \put(115, 10){\tiny $m$}
\put(140, 10){\tiny $m+1$}
\put(187, 10){\tiny $\ell$}
%
%
\put(20, 0){\circle{10}}
\put(45, 2){\line(-1, 0){20}}
\put(45, -2){\line(-1, 0){20}}
\put(30, -3.2){$<$}
\put(50, 0){\circle{10}}
\put(55, 0){\line(1, 0){10}}
\put(65, -1){...}
\put(75, 0){\line(1, 0){10}}
\put(90, 0){\circle{10}}
\put(95, 0){\line(1, 0){20}}
{\color{gray} \put(120, 0){\circle*{10}}}
\put(125, 0){\line(1, 0){20}}
\put(150, 0){\circle{10}}
\put(155, 0){\line(1, 0){10}}
\put(165, -1){...}
\put(175, 0){\line(1, 0){10}}
\put(190, 0){\circle{10}}
\end{picture}
\caption{Degenerate quantum group of type $B$}
\label{fig:B}

\begin{picture}(200, 30)(10,0)
\put(17, 10){\tiny $1$}
\put(80, 10){\tiny $m-1$}
 \put(115, 10){\tiny $m$}
\put(140, 10){\tiny $m+1$}
\put(187, 10){\tiny $\ell$}
%
%
\put(20, 0){\circle{10}}
\put(45, 2){\line(-1, 0){20}}
\put(45, -2){\line(-1, 0){20}}
\put(30, -3.2){$>$}
\put(50, 0){\circle{10}}
\put(55, 0){\line(1, 0){10}}
\put(65, -1){...}
\put(75, 0){\line(1, 0){10}}
\put(90, 0){\circle{10}}
\put(95, 0){\line(1, 0){20}}
{\color{gray} \put(120, 0){\circle*{10}}}
\put(125, 0){\line(1, 0){20}}
\put(150, 0){\circle{10}}
\put(155, 0){\line(1, 0){10}}
\put(165, -1){...}
\put(175, 0){\line(1, 0){10}}
\put(190, 0){\circle{10}}
\end{picture}
\caption{Degenerate quantum group of type $C$}
\label{fig:C}

\begin{picture}(200, 70)(10,-25)
\put(17, 28){\tiny $1$}
\put(17, -33){\tiny $2$}
\put(47, 10){\tiny $3$}
\put(80, 10){\tiny $m-1$}
 \put(115, 10){\tiny $m$}
\put(140, 10){\tiny $m+1$}
\put(187, 10){\tiny $\ell$}
%
%
\put(20, 20){\circle{10}}
\put(20, -20){\circle{10}}
\put(50, 5){\line(-5, 3){26}}
\put(50, -5){\line(-5, -3){26}}
%
%
\put(50, 0){\circle{10}}
\put(55, 0){\line(1, 0){10}}
\put(65, -1){...}
\put(75, 0){\line(1, 0){10}}
\put(90, 0){\circle{10}}
\put(95, 0){\line(1, 0){20}}
{\color{gray} \put(120, 0){\circle*{10}}}
\put(125, 0){\line(1, 0){20}}
\put(150, 0){\circle{10}}
\put(155, 0){\line(1, 0){10}}
\put(165, -1){...}
\put(175, 0){\line(1, 0){10}}
\put(190, 0){\circle{10}}
\end{picture}
\caption{Degenerate quantum group of type $D$}
\label{fig:D}
\end{figure}

Consider the generalised Dynkin diagram of type $X$, which has $\ell$ nodes with the $m$-th one coloured grey. The following three subdiagrams are particularly relevant for our discussion below: the subdiagram on the left side of the grey note, which is a Dynkin diagram of type $X_{m-1}$;
the subdiagram on the right side of the grey note, which is a Dynkin diagram of type $A_{\ell-m}$;
and the subdiagram consisting of the grey node and its two neighbours.  

For the purpose of illustrating the general ideas, we assume that 
$m\ge 3$ if $X=B, C$, and $m\ge 4$ if $X=D$.
Then the degenerate quantum group associated with the generalised Dynkin diagram $X$ is generated  by $\ell$ sets of generators $e_i, f_i, k_i^{\pm 1}$, each set corresponding to a node in the diagram, such that
\begin{enumerate}
\item[(a)] the generators $e_i, f_i, k_i^{\pm 1}$ commute with $e_j, f_j, k_j^{\pm 1}$  if the $i$-th and $j$-th nodes are not directly connected;
\item[(b)] $\{e_i, f_i, k_i^{\pm 1}\mid 1\le i< m\}$ generates the quantum group $\U_q(X_{m-1})$;
\item[(c)] $\{e_i, f_i, k_i^{\pm 1}\mid m+1\le i\le \ell\}$ generates the quantum group $\U_p(A_{\ell-m})$ with $p=-q^{-1}$;
\item[(d)]  $\{e_i, f_i, k_i^{\pm 1}\mid i=m, m\pm 1\}$ generates the degenerate quantum group $\U_q(\fsl_{2, 2})$; 
\end{enumerate}
where,  if $\ell=m$, we replace $(d)$ by
\begin{enumerate}
\item[(d$^\prime$)]  $\{e_i, f_i, k_i^{\pm 1}\mid i=m-1, m\}$ generates the degenerate quantum group $\U_q(\fsl_{2, 1})$. 
\end{enumerate}
Slight modifications of the above are needed for small $m$, which we will discuss in 
a future work, where we will develop a systematic theory of degenerate quantum groups of all finite and affine Kac-Moody types.

\begin{remark} Even though the generalised Dynkin diagrams given here formally look the same as the  Dynkin diagrams of the classical series of Lie superalgebras, they have totally different meanings from the latter.
\end{remark}

\subsection{Similarities with the quantum general linear supergroup}\label{sect:duality}

It is clear that the degenerate quantum general linear group $\U_q(\gl_{m, n})$ studied here is very different from the usual quantum general linear group $\U_q(\gl_{m+n})$ \cite{J2} as Hopf algebras.  It is also quite different from the quantum general linear supergroup $\U_q(\gl_{m|n})$, as the latter is a Hopf superalgebra.  In fact,  Remark \ref{rem:no-deform} implies that $\U_q(\gl_{m, n})$ is not the ``deformation quantisation" of any universal enveloping algebra.
Nevertheless, there are many similarities between $\U_q(\gl_{m, n})$ and $\U_q(\gl_{m|n})$. For example, their definitions both require quartic Serre relations, which differ only  in details;  the parabolic decomposition of $\U_q(\gl_{m, n})$ given in Proposition \ref{prop:para} resembles that of $\U_q(\gl_{m|n})$ given in \cite{Z93};  and
the parametrisations of their finite dimensional simple modules are also similar (see Remark \ref{rem:sim-super}).

It will be very interesting to determine whether there  is a precise connection between the Hopf algebra $\U_q(\gl_{m, n})$ and Hopf superalgebra $\U_q(\gl_{m|n})$.
We hope to investigate this in a future work. This may require a change of the foundation to study Hopf algebras over braided tensor categories \cite{M},  so that Hopf algebras and Hopf superalgebras are put on equal footing.
If any connection exists  between $\U_q(\gl_{m, n})$ and $\U_q(\gl_{m|n})$,  it is likely to be in the form of the quantum correspondences studied in \cite{XZ, Z92, Z97}.  

\appendix
\section{Proof of Lemma \ref{lem:Serre-Q}}

We consider the Hopf algebra $\widetilde{\U}_q(\gl_{m, n})$ defined in Lemma \ref{lem:hopf-mn}.
Here $E_{m-1, m+2}$, $E_{m+2, m-1}$, $E_{m-1, m+1}$ and $E_{m+1, m-1}$ are all elements in
$\widetilde{\U}_q(\gl_{m, n})$, which are defined immediately before the statement of  Lemma \ref{lem:Serre-Q}. The results obtained in this appendix are used in the proof of Lemma \ref{lem:Serre-Q}.

\subsection{Some commutation relations}\label{sect:commutaions}
We have the following result.
\begin{lemma}\label{lem:A1}
The following relations hold in $\widetilde{\U}_q(\gl_{m, n})$.
\begin{eqnarray}
&{[f_m, E_{m-1, m+1}]} =-e_{m-1} k_m q_m^{-1}, \label{eq:lem1}\\
&{[f_{m-1}, E_{m-1, m+1}]}
=  e_m k_{m-1}^{-1}, \label{eq:lem2}\\
&{[f_m, E_{m-1, m+2}]} =0, \label{eq:lem3}\\
&{[f_{m-1}, E_{m-1, m+2}]}
= E_{m, m+2} k_{m-1}^{-1}, \label{eq:lem4}\\
&{[f_{m+1}, E_{m-1, m+2}]}
= - E_{m-1, m+1} k_{m+1}q_{m+1}^{-1}, \label{eq:lem5}
\end{eqnarray}
where $
E_{m, m+2}:=e_m  e_{m+1} - q_{m+1}^{-1} e_{m+1} e_m.
$
\end{lemma}
\begin{proof}
The following computation proves all the relations except the fourth one.
\[
\begin{aligned}
-[f_m, E_{m-1, m+1}]
&= e_{m-1} \frac{k_m- k_m^{-1}}{q_m-q_m^{-1}}
	- q_m^{-1} \frac{k_m- k_m^{-1}}{q_m-q_m^{-1}} e_{m-1}\\
&=e_{m-1} k_m q_m^{-1}, 
\end{aligned}
\]
\[
\begin{aligned}
-[f_{m-1}, E_{m-1, m+1}]
&=	\frac{k_{m-1}- k_{m-1}^{-1}}{q_{m-1}-q_{m-1}^{-1}} e_m - 	q_m^{-1} e_m
\frac{k_{m-1}- k_{m-1}^{-1}}{q_{m-1}-q_{m-1}^{-1}}\\
&= - e_m k_{m-1}^{-1}, 
\end{aligned}
\]
\[
\begin{aligned}
-[f_m, E_{m-1, m+2}]
&=e_{m-1} k_m q_m^{-1} e_{m+1} - q_{m+1}^{-1} e_{m+1} e_{m-1} k_m q_m^{-1}\\
&=0, 
\end{aligned}
\]
\[
\begin{aligned}
-[f_{m+1}, E_{m-1, m+2}]
&=  E_{m-1, m+1} \frac{k_{m+1}- k_{m+1}^{-1}}{q_{m+1}-q_{m+1}^{-1}}  -  q_{m+1}^{-1} \frac{k_{m+1}- k_{m+1}^{-1}}{q_{m+1}-q_{m+1}^{-1}} E_{m-1, m+1}\\
&= E_{m-1, m+1} k_{m+1}q_{m+1}^{-1}.
\end{aligned}
\]
To prove the fourth relation in the lemma, we note that
\[
E_{m-1, m+2}=e_{m-1} E_{m, m+2}  - q_m^{-1} E_{m, m+2} e_{m-1},
\]
which  immediately follows from the definition of $E_{m-1, m+2}$. Thus
\[
\begin{aligned}
-[f_{m-1}, E_{m-1, m+2}]
&=  \frac{k_{m-1}- k_{m-1}^{-1}}{q_{m-1}-q_{m-1}^{-1}} E_{m, m+2} - q_m^{-1} E_{m, m+2}\frac{k_{m-1}- k_{m-1}^{-1}}{q_{m-1}-q_{m-1}^{-1}}\\
&= - E_{m, m+2} k_{m-1}^{-1}.
\end{aligned}
\]
This completes the proof.
\end{proof}

\subsection{Proof of \eqref{eq:Q-2} and \eqref{eq:Q-3}}\label{sect:proof-Q}
\begin{proof} Let us prove \eqref{eq:Q-2}.
This is done by straightforward calculations, which, however, are very lengthy, thus are separated into smaller parts.

(a). Calculation of  $\Delta(E_{m-1,m+1})$.
\begin{eqnarray*}
\begin{aligned}
\Delta(E_{m-1,m+1})
&=E_{m-1, m+1}\otimes k_{m-1}k_m +1\otimes E_{m-1, m+1}\\
&+(1-q^{-2})e_m\otimes e_{m-1}k_m.
\end{aligned}
\end{eqnarray*}

(b). Calculation of $\Delta(E_{m-1,m+2})$.

We have $\Delta(E_{m-1,m+2})=\Delta(E_{m-1,m+1}) \Delta(e_{m+1})
-p^{-1} \Delta(e_{m+1})\Delta(E_{m-1,m+1})$.
The two terms will be calculated separately by using the result of part (a).
\begin{eqnarray*}
\begin{aligned}
&\Delta(E_{m-1,m+1})\Delta(e_{m+1})\\
&= E_{m-1,m+1} e_{m+1}\otimes k_{m-1} k_m k_{m+1} + 1\otimes E_{m-1,m+1} e_{m+1}\\
&+ e_{m+1}\otimes E_{m-1,m+1} k_{m+1} +p^{-1} E_{m-1,m+1}\otimes e_{m+1} k_{m-1} k_m\\
&+(1-q^{-2})e_m e_{m+1}\otimes e_{m-1}k_mk_{m+1}\\
&+p^{-1}(1-q^{-2})e_m\otimes e_{m-1}e_{m+1} k_m; \\
&\Delta(e_{m+1})\Delta(E_{m-1,m+1})\\
&=e_{m+1} E_{m-1, m+1}\otimes k_{m-1} k_m k_{m+1} + 1\otimes e_{m+1} E_{m-1, m+1}\\
&+p^{-1} e_{m+1}\otimes E_{m-1, m+1} k_{m+1} +  E_{m-1, m+1}\otimes e_{m+1} k_{m-1} k_m\\
&+ (1-q^{-2}) e_{m+1} e_m\otimes e_{m-1} k_m k_{m+1}\\
&+ (1-q^{-2})  e_m\otimes e_{m-1} e_{m+1} k_m.
 \end{aligned}
\end{eqnarray*}
Combining these results,  we obtain
\[
\begin{aligned}
\Delta(E_{m-1,m+2})
&=E_{m-1,m+2}\otimes k_{m-1} k_m k_{m+1} + 1\otimes E_{m-1,m+2}\\
&\quad + (1-q^2) e_{m+1}\otimes E_{m-1,m+1} k_{m+1} \\
&\quad + (1-q^{-2}) E_{m, m+2} \otimes e_{m-1} k_m k_{m+1}.
 \end{aligned}
\]

(c). Calculation of $\Delta (Q^{+})$.

We have $\Delta (Q^{+})=\Delta (e_m)\Delta(E_{m-1,m+2})-\Delta(E_{m-1,m+2})\Delta (e_m)$.
Let us denote
\[
\begin{aligned}
M&:=E_{m-1,m+2}\otimes k_{m-1} k_m k_{m+1} + 1\otimes E_{m-1,m+2},\\
W&:= (1-q^2) e_{m+1}\otimes E_{m-1,m+1} k_{m+1} \\
&\quad + (1-q^{-2}) E_{m,  m+2} \otimes e_{m-1} k_m k_{m+1}.
 \end{aligned}
\]
Then $\Delta(E_{m-1,m+2})=M+W$ by part (b),   and hence
\begin{eqnarray}\label{eq:Delta-Q}
\Delta (Q^{+})=\Delta (e_m)(M+W)-(M+W)\Delta (e_m).
\end{eqnarray}

Observe that $k_m$ commutes with $E_{m-1, m+2}$, and $ k_{m-1} k_m k_{m+1}$ with
$e_m$. Thus
\begin{eqnarray}\label{eq:M-commu}
\begin{aligned}
\Delta(e_m) M - M \Delta(e_m)&= [e_m, E_{m-1, m+2}]\otimes k_{m-1} k_m^2 k_{m+1} \\
&\quad + 1\otimes [e_m, E_{m-1, m+2}]\\
&= Q^+\otimes k_{m-1} k_m^2 k_{m+1}  + 1\otimes Q^+.
\end{aligned}
\end{eqnarray}

It is straightforward to obtain
\begin{eqnarray*}
\begin{aligned}
\Delta (e_m)W
&=q^{-1}(1-q^{-2})e_m E_{m, m+2}\otimes e_{m-1} k_m^2 k_{m+1}\\
&+(1-q^{-2})E_{m, m+2}\otimes e_m e_{m-1} k_m k_{m+1}\\
&+ (q-q^{-1}) e_m e_{m+1}\otimes E_{m-1, m+1} k_m k_{m+1}\\
&+ (1-q^2)e_{m+1}\otimes e_m E_{m-1, m+1} k_{m+1}, \\
W\Delta (e_m)
&=(1-q^2) e_{m+1}e_m\otimes E_{m-1, m+1} k_m k_{m+1}\\
&+q(q^2-1) e_{m+1}\otimes E_{m-1,m+1}  e_m k_{m+1}\\
&+(q-q^{-1})E_{m, m+2}\otimes e_{m-1}e_m k_m k_{m+1}\\
&+(1-q^{-2})E_{m, m+2}e_m\otimes e_{m-1}k_m^2k_{m+1}.
\end{aligned}
\end{eqnarray*}
These results lead to
\[
\begin{aligned}
&\Delta (e_m)W - W\Delta (e_m)\\
& = (1-q^{-2})(q^{-1}e_m E_{m, m+2}-  E_{m, m+2}e_m)\otimes e_{m-1} k_m^2 k_{m+1}\\
&+ (1-q^2)e_{m+1}\otimes (e_m E_{m-1, m+1} +q  E_{m-1, m+1} e_m)k_{m+1}\\
&+(q-q^{-1})E_{m, m+2}\otimes (q^{-1}e_m e_{m-1} - e_{m-1}  e_m )k_m k_{m+1}\\
&+ (q-q^{-1}) (e_m e_{m+1}+ q e_{m+1} e_m)\otimes E_{m-1, m+1} k_m k_{m+1}\\
\end{aligned}
\]
Since
$
q^{-1}e_m E_{m, m+2}-  E_{m, m+2}e_m = 0$ and $
e_m E_{m-1, m+1} +q  E_{m-1, m+1} e_m=0,
$
the first two terms on the right side vanish independently; by the definitions of
$E_{m, m+2}$ and $E_{m-1, m+1}$,  the last two terms cancel out.  We arrive at
\begin{eqnarray}\label{eq:W-commu}
\Delta (e_m)W - W\Delta (e_m)=0.
\end{eqnarray}

Now using \eqref{eq:M-commu} and \eqref{eq:W-commu} in equation \eqref{eq:Delta-Q},  we immediately obtain
\[
\Delta(Q^+)=Q^+ \otimes  k_{m-1} k_m^2 k_{m+1} +1\otimes Q^+.
\]
This completes the proof of  equation \eqref{eq:Q-2}.

Equation  \eqref{eq:Q-3} can be proved similarly; we omit the details.
\end{proof}

\end{document}